\documentclass[11pt]{article}

\usepackage{enumerate, amsmath, amsthm, amsfonts, amssymb, xy}
\usepackage[margin=1in]{geometry}
\usepackage{algorithm, algorithmic}

\input xy
\xyoption{all}

\newtheorem{theorem}{Theorem}[section]
\newtheorem{proposition}[theorem]{Proposition}
\newtheorem{lemma}[theorem]{Lemma}
\newtheorem{definition}[theorem]{Definition}
\newtheorem{example}[theorem]{Example}
\newtheorem{rmk}[theorem]{Remark}
\newtheorem{corollary}[theorem]{Corollary}

\numberwithin{equation}{section}

\newcommand{\Z}{\mathbb{Z}}

\renewcommand{\phi}{\varphi}
\renewcommand{\emptyset}{\varnothing}

\makeatletter
\def\Ddots{\mathinner{\mkern1mu\raise\p@
\vbox{\kern7\p@\hbox{.}}\mkern2mu
\raise4\p@\hbox{.}\mkern2mu\raise7\p@\hbox{.}\mkern1mu}}
\makeatother

\newcommand{\rank}{\operatorname{rank}}
\newcommand{\Des}{\operatorname{Des}}
\newcommand{\des}{\operatorname{des}}
\newcommand{\exc}{\operatorname{exc}}
\newcommand{\inv}{\operatorname{inv}}

\newcommand{\LdDes}{\operatorname{LdDes}}
\newcommand{\Rexc}{\operatorname{Rexc}}
\newcommand{\CT}{\operatorname{C}}
\newcommand{\cover}{\operatorname{cover}}
\newcommand{\rev}{\operatorname{rev}}

\newcommand{\Expose}{\operatorname{Expose}}

\newcommand{\cef}{\operatorname{cef}}

\title{Ehrhart $h^*$-vectors of hypersimplices}

\author{Nan Li\thanks{Partially supported by the US National Science
Foundation under Grant DMS-0604423.}}

\begin{document}

\maketitle

\begin{abstract}
We consider the Ehrhart $h^*$-vector for the hypersimplex. It is
well-known that the sum of the $h_i^*$ is the normalized volume which
equals an Eulerian numbers. The main result is a proof of a conjecture
by R. Stanley which gives an interpretation of the $h^*_i$ coefficients
in terms of descents and excedances. Our proof is geometric using a
careful book-keeping of a shelling of a unimodular triangulation.
We generalize this result to other closely related polytopes.

\end{abstract}
\section{Introduction}

 Hypersimplices appear
naturally in algebraic and geometric contexts. For example, they can
be considered as moment polytopes for torus actions on Grassmannians
or weight polytopes of the fundamental representations of the
general linear groups $GL_n$. Fix two integers $0<k\le n$. The
$(k,n)$-th \emph{hypersimplex} is defined as follows
$$\overline{\Delta}_{k,n}=\{(x_1,\dots,x_{n})\mid 0\le x_1,\dots,x_{n}\le
1;\, x_1+\cdots +x_{n}=k \},$$ or equivalently,
$$\Delta_{k,n}=\{(x_1,\dots,x_{n-1})\mid 0\le x_1,\dots,x_{n-1}\le
1;\,k-1\le x_1+\cdots +x_{n-1}\le k\}.$$ They can be considered as
the slice of the hypercube $[0,1]^{n-1}$ located between the two
hyperplanes $\sum x_i=k-1$ and $\sum x_i=k$.

For a permutation $w\in\mathfrak{S}_n$, we call $i\in[n-1]$ a
\emph{descent} of $w$, if $w(i)>w(i+1)$. We define $\des(w)$ to be
the number of descents of $w$. We call $A_{k,n-1}$ the Eulerian
number, which equals the number of permutations in
$\mathfrak{S}_{n-1}$ with $\des(w)=k-1$. The following result is
well-known (see for example, \cite[Exercise 4.59 (b)]{EC1}).
\begin{theorem} [Laplace]\label{vol}
 The normalized volume of
$\Delta_{k,n}$ is the Eulerian number $A_{k,n-1}$.
\end{theorem}
Let $S_{k,n}$ be the set of all points $(x_1,\dots,x_{n-1})\in
[0,1]^{n-1}$ for which $x_i<x_{i+1}$ for exactly $k-1$ values of $i$
(including by convention $i=0$). Foata asked whether there is some
explicit measure-preserving map that sends $S_{k,n}$ to
$\Delta_{k,n}$. Stanley \cite{sta1} gave such a map, which gave
a triangulation of the hypersimplex into $A_{k,n-1}$ unit simplices
and provided a geometric proof of Theorem \ref{vol}. Sturmfels
\cite{stu} gave another triangulation of $\Delta_{k,n}$, which
naturally appears in the context of Gr\"obner bases. Lam and Postnikov \cite{lp}
 compared these two triangulations
together with the alcove triangulation and the circuit
triangulation. They showed that these four triangulations are
identical. We call a triangulation of a convex polytope
\emph{unimodular} if every simplex in the triangulation has
normalized volume one. It is clear that the above triangulations of
the hypersimplex are unimodular.

Let $\mathcal{P}\in \Z^{N}$ be any $n$-dimensional integral polytope
(its vertices are given by integers). Then Ehrhart's theorem tells
us that the function
$$i(\mathcal{P},r):=\#(r\mathcal{P}\cap \Z^N)$$
is a polynomial in $r$, and
$$\sum_{r\ge 0}i(\mathcal{P},r)t^r=\frac{h^*(t)}{(1-t)^{n+1}},$$
where $h^*(t)$ is a polynomial in $t$ with degree $\le n$. We call
$h^*(t)$ the \emph{$h^*$-polynomial} of $\mathcal{P}$, and the vector
$(h^*_0,\dots,h^*_n)$, where $h^*_i$ is the coefficient of $t^i$ in
$h^*(t)$,
 is called the \emph{$h^*$-vector} of $\mathcal{P}$.  We know that
 the sum $\sum_{i=0}^{i=n}h^*_i(\mathcal{P})$ equals the normalized
 volume of $\mathcal{P}$.

 Katzman \cite{kat} proved the following formula for the $h^*$-vector of
 the hypersimplex $\Delta_{k,n}$. In particular, we see that
 $\sum_{i=0}^{i=n}h^*_i(\Delta_{k,n})=A_{k,n-1}$. Write $\binom{n}{r}_{\ell}$ to
 denote the coefficient of $t^r$ in
 $(1+t+t^2+\dots+t^{\ell-1})^n$. Then the $h^*$-vector of $\Delta_{k,n}$
 is $(h^*_0(\Delta_{k,n}),\dots,h^*_{n-1}(\Delta_{k,n}))$, where for
$d=0,\dots,n-1$
 \begin{equation}\label{h}
 h^*_d(\Delta_{k,n})=\sum_{i=0}^{k-1}(-1)^i\binom{n}{i}\binom{n}{(k-i)d-i}_{k-i}.
\end{equation}
 Moreover, since all the $h^*_i(\Delta_{k,n})$ are
 nonnegative integers (\cite{sta}) (this is not clear from (\ref{h})),
 it will be interesting to give a
 combinatorial interpretation of the $h^*_i(\Delta_{k,n})$.

 The \emph{half-open hypersimplex} $\Delta'_{k,n}$ is defined as follows. If $k>1$,
 $$\Delta'_{k,n}=\{(x_1,\dots,x_{n-1})\mid 0\le x_1,\dots,x_{n-1}\le
1;\,k-1< x_1+\cdots +x_{n-1}\le k\},$$  and
$$\Delta'_{1,n}=\Delta_{1,n}.$$
We call $\Delta'_{k,n}$ ``half-open" because it is basically the
normal
 hypersimplex with the ``lower" facet removed. From the definitions, it is clear that the volume
 formula and triangulations of the usual hypersimplex $\Delta_{k,n}$
 also work for the half-open hypersimplex $\Delta'_{k,n}$, and it is
 nice that for fixed $n$, the half-open hypersimplices $\Delta'_{k,n}$, for
 $k=1,\dots,n-1$, form a disjoint union of the hypercube
 $[0,1]^{n-1}$. From the following formula for the $h^*$-polynomial of the half-open
 hypersimplices, we can compute the $h^*$-polynomial of the usual
 hypersimplices inductively. Also, we can compute its Ehrhart
 polynomial.

 For a permutation $w$,
 we call $i$ an \emph{excedance} of $w$ if $w(i)>i$ (a \emph{reversed excedance}
 if $w(i)<i$). We denote by $\exc(w)$ the number of excedances of
 $w$. The main theorems of the paper are the following.
\begin{theorem}\label{exc}The $h^*$-polynomial of the half-open hypersimplex $\Delta'_{k,n}$ is
given by,
$$\sum_{\substack{w\in\, \mathfrak{S}_{n-1}\\ \exc(w)=k-1}}t^{\des(w)}.$$
\end{theorem}
We prove this theorem first by a
generating function method (in Section 2) and second by a geometric
method, i.e., giving a shellable triangulation of the hypersimplex
(in Sections 3, 4 and 5).

We can define a different shelling order on the triangulation of
$\Delta'_{k,n}$, and get another expression of  its $h^*$-polynomial using descents and a
new permutation statistic called \emph{cover} (see its definition in Lemma~\ref{defcover}).
\begin{theorem} \label{cover}The $h^*$-polynomial of $\Delta'_{k,n}$ is
$$\sum_{\substack{w\in \mathfrak{S}_{n-1}\\
\des(w)=k-1}}t^{\cover(w)}.$$
\end{theorem}
Combine Theorem~\ref{cover} with Theorem \ref{exc}, we have the equal distribution of $(\exc, \des)$ and
$(\des,\cover)$:
\begin{corollary}\label{equal}
$$\sum_{w\in\mathfrak{S}_n}t^{\des(w)}x^{\cover(w)}=\sum_{w\in\mathfrak{S}_n}t^{\exc(w)}x^{\des(w)}.$$
\end{corollary}
Finally, we study the generalized hypersimplex $\Delta_{k,\alpha}$ (Section 7). This polytope is related to algebras
of Veronese type. For example, it is known  \cite{hibi} that every
algebra of Veronese type coincides with the Ehrhart ring of a
polytope $\Delta_{k, \alpha}$. We can extend
this second shelling to the generalized hypersimplex $\Delta'_{k,\alpha}$ (defined in (\ref{general})),  and express its $h^*$-polynomial
in terms of a colored version of descents and covers (see Theorem~\ref{color}).

\section{Proof of Theorem \ref{exc} by generating functions}
Here is a proof of this theorem using generating functions.
\begin{proof}
Suppose we can show that
\begin{equation}\label{goal}
\sum_{r\ge 0}\sum_{k\ge 0}\sum_{n\ge 0}i(\Delta'_{k+1,n+1},r)u^ns^kt^r=\sum_{n\ge 0}\sum_{\sigma\in \mathfrak{S}_n}t^{\des(\sigma)}s^{\exc(\sigma)}\frac{u^n}{(1-t)^{n+1}}.
\end{equation}
By considering the coefficient of $u^ns^k$ in (\ref{goal}), we have
$$\sum_{r\ge 0}i(\Delta'_{k+1,n+1},r)t^r=(1-t)^{-(n+1)}\left(\sum_{\substack{w\in\, \mathfrak{S}_{n}\\
\exc(w)=k}}t^{\des(w)}\right),$$
which implies Theorem \ref{exc}. By the following equation due to Foata and Han \cite[Equation (1.15)]{fh},
$$\sum_{n\ge 0}\sum_{\sigma\in \mathfrak{S}_n}t^{\des(\sigma)}s^{\exc(\sigma)}\frac{u^n}{(1-t)^{n+1}}=
\sum_{r\ge 0}t^r\frac{1-s}{(1-u)^{r+1}(1-us)^{-r}-s(1-u)},$$  we only
need to show that
$$\sum_{k\ge 0}\sum_{n\ge 0}i(\Delta'_{k+1,n+1},r)u^ns^k=\frac{1-s}{(1-u)^{r+1}(1-us)^{-r}-s(1-u)}.$$

By the definition of the half-open hypersimplex, we have, for any $r\in\Z_{\ge 0}$,
$$r\Delta'_{k+1,n+1}=\{(x_1,\dots,x_n)\mid 0\le x_1,\dots,x_n\le r, rk+1\le x_1+\dots+x_n\le (k+1)r\},$$
if $k>0$, and for $k=0$,
$$r\Delta'_{1,n+1}=\{(x_1,\dots,x_n)\mid 0\le x_1,\dots,x_n\le r, 0\le x_1+\dots+x_n\le r\}.$$
So \begin{equation}\label{k>0}
i(\Delta'_{k+1,n+1},r)=([x^{kr+1}]+\cdots+[x^{(k+1)r}])\left(\frac{1-x^{r+1}}{1-x}\right)^{n},
\end{equation}
 if $k>0$, and when $k=0$, we have
 \begin{equation}\label{k=0}
 i(\Delta'_{1,n+1},r)=([x^0]+[x]+\cdots+[x^{r}])\left(\frac{1-x^{r+1}}{1-x}\right)^{n}.
 \end{equation} Notice that the case of $k=0$ is different from $k>0$ and $i(\Delta'_{1,n+1},r)$ is obtained by evaluating $k=0$ in (\ref{k>0}) plus an extra term $[x^{0}]\left(\frac{1-x^{r+1}}{1-x}\right)^{n}$. Since the coefficient of $x^k$  of a function
$f(x)$ equals the constant term of $\frac{f(x)}{x^k}$, we have
\begin{align*}
([x^{kr+1}]+\cdots+[x^{(k+1)r}])\left(\frac{1-x^{r+1}}{1-x}\right)^{n}&=[x^0]\left(\frac{1-x^{r+1}}{1-x}\right)^{n}(x^{-kr-1}+\cdots+x^{-(k+1)r})\\
&=[x^{kr}]\left(\frac{1-x^{r+1}}{1-x}\right)^{n}(x^{-kr-1}+\cdots+x^{-(k+1)r})x^{kr}\\
&=[x^{kr}]\frac{(1-x^r)(1-x^{r+1})^n}{(1-x)^{n+1}x^r}.
\end{align*}
So we have, for $k>0$,
\begin{align*}
\sum_{n\ge 0}i(\Delta'_{k+1,n+1},r)u^n=& \sum_{n\ge 0}[x^{kr}]\frac{(1-x^r)(1-x^{r+1})^n}{(1-x)^{n+1}x^r}u^n\\
& =[x^{kr}]\frac{(1-x^r)}{(1-x)x^r}\sum_{n\ge 0}\left(\frac{(1-x^{r+1})u}{1-x}\right)^n\\
&=[x^{kr}]\frac{x^r-1}{x^{r}(u-ux^{r+1}-1+x)}.
\end{align*}
 For $k=0$, based on the difference between (\ref{k>0}) and (\ref{k=0}) observed above, we have:
\begin{align*}
\sum_{n\ge 0}i(\Delta'_{1,n+1},r)u^n=& \sum_{n\ge
0}[x^{0}]\frac{(1-x^r)(1-x^{r+1})^n}{(1-x)^{n+1}x^r}u^n
+\sum_{n\ge 0}[x^0]\left(\frac{1-x^{r+1}}{1-x}\right)^{n}u^n\\
&=\left([x^{0}]\frac{x^r-1}{x^{r}(u-ux^{r+1}-1+x)}\right)+\frac{1}{1-u}.
\end{align*}

So

$$\sum_{k\ge 0}\sum_{n\ge 0}i(\Delta'_{k+1,n+1},r)u^ns^k=\left(
\sum_{k\ge
0}[x^{kr}]\frac{x^r-1}{x^{r}(u-ux^{r+1}-1+x)}s^k\right)+\frac{1}{1-u}.$$
Let $y=x^r$. We have
$$\sum_{k\ge 0}\sum_{n\ge 0}i(\Delta'_{k+1,n+1},r)u^ns^k=\sum_{k\ge 0}[x^{kr}]\frac{y-1}{y(u-uxy-1+x)}s^k+\frac{1}{1-u}.$$
Expand $\frac{y-1}{y(u-uxy-1+x)}$ in powers of $x$, we have

\begin{align*}\frac{y-1}{y(u-uxy-1+x)}&=\frac{y-1}{y}\cdot \frac{1}{u-1-(uxy-x)}\\
&=\frac{y-1}{y(u-1)}\cdot\frac{1}{1-\frac{x(uy-1)}{u-1}}\\
&=\frac{1-y}{y(1-u)}\sum_{i\ge 0}\left(\frac{(1-uy)x}{1-u}\right)^i.
\end{align*}
Since we only want the coefficient of $x^i$ such that $r$ divides
$i$, we get
\begin{align*}\frac{1-y}{y(1-u)}\sum_{j\ge 0}\left(\frac{(1-uy)x}{1-u}\right)^{rj}
&=\frac{1-y}{y(1-u)}\cdot \frac{1}{1-\frac{(1-uy)^rx^r}{(1-u)^r}}\\
&=\frac{1-y}{y(1-u)}\cdot \frac{(1-u)^r}{(1-u)^r-(1-uy)^rx^r}\\
&=\frac{(1-u)^{r-1}(1-y)}{y(1-u)^r-y^2(1-yu)^r}.
\end{align*}

So
$$\sum_{k\ge 0}\sum_{n\ge 0}i(\Delta'_{k+1,n+1},r)u^ns^k=\left(
\sum_{k\ge
0}s^k[y^{k}]\frac{(1-u)^{r-1}(1-y)}{y(1-u)^r-y^2(1-yu)^r}\right)+\frac{1}{1-u}.$$

 To remove all negative powers of $y$, we do the following
expansion
\begin{align*}
\frac{(1-u)^{r-1}(1-y)}{y(1-u)^r-y^2(1-yu)^r}=&\frac{1-y}{(1-u)y}\cdot\frac{1}{1-\frac{y(1-yu)^r}{(1-u)^r}}\\
=&\sum_{i\ge 0}\left(\frac{y^{i-1}(1-uy)^{ri}}{(1-u)^{ri+1}}-\frac{y^{i}(1-uy)^{ri}}{(1-u)^{ri+1}}\right)\\
=&\frac{1}{1-u}y^{-1}+\text{ nonnegative powers of }y.
\end{align*}
Notice that $\sum_{k\ge
0}s^k[y^{k}]\frac{(1-u)^{r-1}(1-y)}{y(1-u)^r-y^2(1-yu)^r}$ is
obtained by taking the sum of nonnegative powers of $y$ in
$\frac{(1-u)^{r-1}(1-y)}{y(1-u)^r-y^2(1-yu)^r}$ and replacing $y$ by
$s$. So  $$\sum_{k\ge
0}s^k[y^{k}]\frac{(1-u)^{r-1}(1-y)}{y(1-u)^r-y^2(1-yu)^r}
=\frac{(1-u)^{r-1}(1-s)}{s(1-u)^r-s^2(1-su)^r}-\frac{1}{s(1-u)}.$$
Therefore,
\begin{align*}
 \sum_{k\ge 0}\sum_{n\ge 0}i(\Delta'_{k+1,n+1},r)u^ns^k
=&\frac{(1-u)^{r-1}(1-s)}{s(1-u)^r-s^2(1-su)^r}-\frac{1}{s(1-u)}+\frac{1}{1-u}\\
=&\frac{1-s}{(1-u)^{r+1}(1-us)^{-r}-s(1-u)}.\qedhere
\end{align*}
\end{proof}
\section{Background}

\subsection{Shellable triangulation and the $h^*$-polynomial}
Let $\Gamma$ be a triangulation of an $n$-dimensional polytope
$\mathcal{P}$, and let $\alpha_1,\dots,\alpha_s$ be an ordering of the simplices (maximal faces)
of $\Gamma$.  We call $(\alpha_1,\dots,\alpha_s)$ a
\emph{shelling} of $\Gamma$ \cite{sta}, if for each $2\le i\le s$,
$\alpha_i\cap (\alpha_1\cup\dots\cup\alpha_{i-1})$ is a union of
facets ($(n-1)$-dimensional faces) of $\alpha_i$. For example, (ignore
the letters $A$, $B$, and $C$ for now) $\Gamma_1$ is a shelling,
while any order starting with $\Gamma_2$ cannot be a shelling.
$$\Gamma_1: 
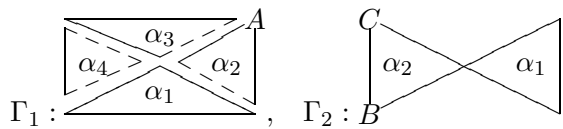
\xy 0;/r.3pc/: (0,0)*{}="a";
 (20,0)*{}="b";
 (20,10)*{A}="c";
 (0,10)*{}="d";
 (10,5)*{}="e";
 (12,5.5)*{}="f";
 (20,1)*{}="g";
 (10,6)*{}="h";
 (18,10)*{}="i";
 (0,2)*{}="l";
 (8,5.5)*{}="j";
 (0,9)*{}="k";
 "a"; "b"**\dir{-};
 "d"; "i"**\dir{-};
  "a"; "e"**\dir{-};
   "e"; "b"**\dir{-};
    "d"; "h"**\dir{-};
     "k"; "l"**\dir{-};
     "c"; "f"**\dir{-};
     "g"; "c"**\dir{-};
     "g"; "f"**\dir{--};
     "h"; "i"**\dir{--};
     "k"; "j"**\dir{--};
     "l"; "j"**\dir{--};
     (10,2)*{\alpha_1};
     (17,5)*{\alpha_2};
     (10,8)*{\alpha_3};
     (3,5)*{\alpha_4};
\endxy,\,\,\,\,\,\Gamma_2: \xy 0;/r.3pc/: (0,0)*{B}="a";
 (0,10)*{C}="b";
 (20,10)*{}="c";
 (20,0)*{}="d";
   "a"; "b"**\dir{-};
     "c"; "d"**\dir{-};
     "b"; "d"**\dir{-};
     "a"; "c"**\dir{-};
     (17,5)*{\alpha_1};
     (3,5)*{\alpha_2};
\endxy$$
An equivalent condition (see e.g., \cite{sta4}) for a shelling is
that every simplex has a \emph{unique minimal non-face}, where by
a ``non-face", we mean a face that has not appeared in previous
simplices. For example, for $\alpha_2\in\Gamma_1$, the vertex $A$ is
its unique minimal non-face, while for $\alpha_2\in\Gamma_2$, both
$B$ and $C$ are minimal and have not appeared before $\alpha_2$. We
call a triangulation with a shelling a \emph{shellable
triangulation}. Given a shellable triangulation $\Gamma$ and a
simplex $\alpha\in\Gamma$, define the \emph{shelling number} of
$\alpha$ (denoted by $\#(\alpha)$) to be the number of facets shared
by $\alpha$ and some simplex preceding $\alpha$ in the shelling order.
 For example, in
$\Gamma_1$, we have
$$\#(\alpha_1)=0,\,\#(\alpha_2)=1,\,\#(\alpha_3)=1,\,\#(\alpha_4)=2.$$
The benefit of having a shelling order for Theorem \ref{exc} comes
from the following result.
\begin{theorem}[{\cite{sta} Shelling and Ehrhart polynomial}]
\label{Shelling and Ehrhart polynomial}Let $\Gamma$ be a unimodular
shellable triangulation of an $n$-dimensional polytope
$\mathcal{P}$. Then
$$\sum_{r\ge 0}i(\mathcal{P},r)t^r=(\sum_{\alpha\in\,
\Gamma}t^{\#(\alpha)})(1-t)^{-(n+1)}.$$
\end{theorem}To be self-contained, we include a short proof here.
\begin{proof}Given a shellable triangulation, we get a
partition of $\mathcal{P}$: for any simplex $\alpha$, let $\alpha'\subset
\alpha$ be obtained from $\alpha$ by removing the facets that
$\alpha$ shares with the simplices preceding it in the shelling order.
 The fact that $\Gamma$ is shellable will
guarantee that this is a well-defined partition, i.e., there is no
overlap and no missing area. So we can sum over all the parts to
compute $i(\mathcal{P},r)$ (the number of integer points of
$r\mathcal{P}$). If $\mathcal{F}$ is a $d$-dimensional simplex, then
$$\sum_{r\ge 0}i(\mathcal{F},r)t^r=\frac{1}{(1-t)^{d+1}}.$$ Since
the triangulation is unimodular, $\alpha$ is an $n$-dimensional
simplex. Let $k:=\#(\alpha)$. Since $\alpha'$ is obtained from $\alpha$
by removing $k$ simplices of dimension $n-1$ from $\alpha$, the inclusion-exclusion
formula implies that
$$\sum_{r\ge 0}i(\mathcal{\alpha'},r)t^r=(1-t)^{-(n+1)}
\left(\sum_{i=0}^{k}(-1)^i\binom{k}{i}(1-t)^i\right)=\frac{t^{\#(\alpha)}}{(1-t)^{n+1}}.$$\qedhere

\end{proof}For example, $\Gamma_1$ in the previous example gives us a partition as shown
above, and we have
$$\sum_{r\ge
0}i(\alpha'_1,r)t^r=\frac{1}{(1-t)^{3}},\,\,\sum_{r\ge
0}i(\alpha'_2,r)t^r=\frac{1}{(1-t)^{3}}-\frac{1}{(1-t)^{2}}=\frac{t}{(1-t)^{3}},$$and
$$\sum_{r\ge
0}i(\alpha'_4,r)t^r=\frac{1}{(1-t)^{3}}-2\frac{1}{(1-t)^{2}}+\frac{1}{(1-t)}=\frac{t^2}{(1-t)^{3}}.$$
\subsection{Excedances and descents}
Let $w\in
\mathfrak{S}_n$. Define its \emph{standard representation of cycle
notation} to be a cycle notation of $w$ such that the first element
in each cycle is its largest element and the cycles are ordered with
their largest elements increasing. We define the \emph{cycle type}
of $w$ to be the composition of $n$: $\CT(w)=(c_1,\dots,c_k)$ where
$c_i$ is the length of the $i$th cycle in its standard
representation. The \emph{Foata map} $F\colon w\rightarrow
\hat{w}$ maps $w$  to $\hat{w}$  obtained from $w$ by removing parentheses
from the standard representation of $w$. For example, consider a
permutation $w\colon [5]\rightarrow [5]$ given by $w(1)=5$,
$w(2)=1$, $w(3)=4$, $w(4)=3$ and $w(5)=2$ or in one line notation
$w=51432$. Its standard representation of cycle notation is
$(43)(521)$, so $\hat{w}=43521$. The inverse Foata map $F^{-1}\colon
\hat{w}\rightarrow w$  allows us to go back from $\hat{w}$ to $w$ as
follows: first insert a left parenthesis before every left-to-right
maximum and then close each cycle by inserting a right parenthesis
accordingly. In the example, the left-to-right maximums of
$\hat{w}=43521$ are $4$ and $5$, so we get back $(43)(521)$. Based
on the Foata map, we have the following result for the equal
distribution of excedances and descents.
\begin{theorem}[Excedances and descents]\label{excdes}The number of permutations
in $\mathfrak{S}_n$ with $k$ excedances equals  the number of
permutations in $\mathfrak{S}_n$ with $k$ descents.
\end{theorem}
\begin{proof}First notice that we can change a permutation with $k$
excedances $u$ to a permutation $w$ with $k$ reverse excedances and
vice versa by applying a \emph{reverse map}: first reverse the
letters by changing $u(i)$ to $n+1-u(i)$, then reverse the positions
by defining $n+1-u(i)$ to be $w(n+1-i)$. This way, $i$ is an
excedance of $u$ if and only if $n+1-i$ is a reverse excedance of
$w$. Then the hard part is the connection between descents and
reverse excedances, which will involve the Foata map.

Let $\hat{w}$ be a permutation with $k$ descents
$\{(\hat{w}(i_1),\hat{w}(i_1+1)),\dots,(\hat{w}(i_k),\hat{w}(i_k+1))\}$
with $\hat{w}(i_s)>\hat{w}(i_s+1)$ for $s=1,\dots,k$. We want to
find its preimage $w$ in the above map. After inserting parentheses
in $\hat{w}$, each pair $(\hat{w}(i_s),\hat{w}(i_s+1))$ lies in the
same cycle. So in $w$, we have
$w(\hat{w}(i_s))=\hat{w}(i_s+1)<\hat{w}(i_s)$, therefore,
$\hat{w}(i_s)$ is a reverse excedance of $w$. We also have that each
reverse excedance of $w$ corresponds to a descent in $\hat{w}$ by the definition
of the Foata map. This
finishes the proof.
\end{proof}For
example, to change a permutation with three excedances to a
permutation with three descents, first
$$\dot{4}\dot{3}2\dot{5}1\xrightarrow{(6-)}\dot{2}\dot{3}4\dot{1}5\xrightarrow[\text{position}]{\text{reverse}}5\dot{1}4\dot{3}\dot{2},$$
 changes an excedance in position $i$ to a reverse excedance in
position $6-i$, and then
$$5\dot{1}4\dot{3}\dot{2}\xrightarrow[\text{ of cycle structure}]{\text{standard representation}} (43)(521)
\xrightarrow{\text{remove parentheses}}
\underline{43}\,\underline{521},$$  changes a reverse excedance in
position $i$ to a descent with the first letter $i$. The above two
maps are both reversible.
\subsection{Triangulation of the hypersimplex}

We start form a unimodular triangulation $\{t_w\mid w\in \mathfrak{S}_n\}$ of the hypercube, where
$$t_w=\{(y_1,\dots,y_n)\in [0,1]^n\mid 0\le y_{w_1}\le y_{w_n}\le
\dots\le y_{w_n}\}.$$ It is easy to see that $t_w$ has the following
$n+1$ vertices: $v_0=(0,\dots,0)$, and $v_i=(y_1,\dots,y_n)$ given
by $y_{w_1}=\dots=y_{w_{n-i}}=0$ and
$y_{w_{n-i+1}}=\dots=y_{w_n}=1$. It is clear that
$v_{i+1}=v_{i}+e_{w_{n-i}}$. Now define the following map $\phi$ (\cite{sta1},\cite{lp})
that maps $t_w$ to $s_w$, sending $(y_1,\dots,y_n)$ to $(x_1,\dots,x_n)$,
where
\begin{equation}\label{stanmap}
x_i=\begin{cases}
y_i-y_{i-1}, & \text{ if } (w^{-1})_i>(w^{-1})_{i-1},\\
1+y_i-y_{i-1},& \text{ if } (w^{-1})_i<(w^{-1})_{i-1},
\end{cases}
\end{equation}
where we set $y_0=0$. For each point $(x_1,\dots,x_n)\in s_w$, set
$x_{n+1}=k+1-(x_1+\dots+x_n)$. Since $v_{i+1}$ and $v_i$ only differ
in $y_{w_{n-i}}$, by (\ref{stanmap}), $\phi(v_i)$ and
$\phi(v_{i+1})$ only differ in $x_{w_{n-i}}$ and $x_{w_{n-i}+1}$. More explicitly, we have

\begin{lemma}\label{move}Denote $w_{n-i}$ by $r$. For $\phi(v_i)$, we have $x_{r}x_{r+1}=01$ and  for
$\phi(v_{i+1})$, we have $x_{r}x_{r+1}=10$. In other words, from
$\phi(v_{i})$ to $\phi(v_{i+1})$, we move a 1 from the $(r+1)$th
coordinate forward by one coordinate.
\end{lemma}
\begin{proof}  First, we want to show that for $\phi(v_i)$, we have
 $x_{r}=0$ and $x_{r+1}=1$. We need to look at the segment
 $y_{r-1}y_{r}y_{r+1}$,
  of $v_i$. We know that
  $y_{r}=0$, so there are four cases for $y_{r-1}y_{r}y_{r+1}$: 000, 001, 100, 101.
  If $y_{r-1}y_{r}y_{r+1}=000$ for $v_i$, then
$y_{r-1}y_{r}y_{r+1}=010$ for $v_{i+1}$. Therefore,
$w^{-1}_{r-1}<w^{-1}_{r}>w^{-1}_{r+1}$. Then by (\ref{stanmap}), we
have $x_{r}x_{r+1}=01$. Similarly, we can check in the other three
cases that $x_{r}x_{r+1}=01$ for $\phi(v_i)$.

Similarly, we can check the four cases for $y_{r-1}y_{r}y_{r+1}$:
010, 011, 110, 111 in $\phi(v_{i+1})$ and get $x_{r}x_{r+1}=10$ in
all cases.
\end{proof}

Let $\des(w^{-1})=k$. It follows from Lemma~\ref{move} that the sum of
the coordinates $\sum_{i=1}^{n}x_i$ for each vertex $\phi(v_i)$ of $s_w$
is either $k$ or $k+1$. So we have the triangulation \cite{sta1} of
the hypersimplex $\Delta_{k+1,n+1}$: $\Gamma_{k+1,n+1}=\{s_w\mid w\in \mathfrak{S}_n,\,\des(w^{-1})=k\}$.

Now we consider a graph $G_{k+1,n+1}$ on the set of simplices in the
triangulation of $\Delta_{k+1,n+1}$. There is an
edge between two simplices $s$ and $t$ if and only if they are
adjacent (they share a common facet). We can
represent each vertex of $G_{k+1,n+1}$ by a permutation and describe each edge
of $G_{k+1,n+1}$ in terms of permutations \cite{lp}. We call this new graph $\Gamma_{k+1,n+1}$. It
is clear that $\Gamma_{k+1,n+1}$ is isomorphic to $G_{k+1,n+1}$.
\begin{proposition}[{\cite[Lemma 6.1 and Theorem
7.1]{lp}}]\label{gamma}The graph $\Gamma_{k+1,n+1}$ can be described
as follows: its vertices are permutations $u=u_1\dots u_n\in
\mathfrak{S}_{n}$ with  $\des(u^{-1})=k$. There is an edge between
$u$ and $v$, if and only if one of the following two holds:
\begin{enumerate}
\item (type one edge) $u_i-u_{i+1}\neq \pm 1$ for some
$i\in\{1,\dots,n-1\}$, and $v$ is obtained from $u$ by exchanging
$u_i,u_{i+1}$.
\item (type two edge)
$u_n\neq 1,n$, and $v$ is obtained from $u$ by moving $u_n$ to the
front of $u_1$, i.e., $v=u_nu_1\dots u_{n-1}$; or this holds with
$u$ and $v$ switched.
\end{enumerate}
\end{proposition}

\begin{example} Here is the graph $\Gamma_{3,5}$ for $\Delta'_{3,5}$.
$$\Gamma_{3,5}:  
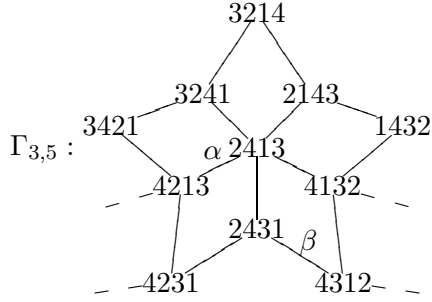
\xy 0;/r.17pc/: (0,0)*{2413}="a";
 (0,25)*{3214}="b";
 (-10,10)*{3241}="c";
 (10,10)*{2143}="d";
 (-27,4)*{3421}="e";
 (27,4)*{1432}="f";
 (-14,-7)*{4213}="g";
 (14,-7)*{4132}="h";
 (0,-15)*{2431}="i";
 (-16,-25)*{4231}="j";
 (16,-25)*{4312}="k";
 "b"; "c"**\dir{-};
 "b"; "d"**\dir{-};
 "c"; "e"**\dir{-};
 "c"; "a"**\dir{-};
 "d"; "a"**\dir{-};
 "d"; "f"**\dir{-};
 "e"; "g"**\dir{-};
 "a"; "g"**\dir{-};
 ?*!/_-2mm/{\alpha};
 "a"; "h"**\dir{-};
 "f"; "h"**\dir{-};
 "g"; "j"**\dir{-};
 "a"; "i"**\dir{-};
 "h"; "k"**\dir{-};
 "i"; "j"**\dir{-};
 "i"; "k"**\dir{-};
  ?*!/_2mm/{\beta};
   (-28,-11)*{}="gg";
 (28,-11)*{}="hh";
 (-30,-27)*{}="jj";
 (30,-27)*{}="kk";
  "jj";"j"**\dir{--};
  "hh";"h"**\dir{--};
 "kk";"k"**\dir{--};
   "gg";"g"**\dir{--};
\endxy$$
In the above graph, the edge $\alpha$ between $u=2413$ and $v=4213$
is  a type one edge with $i=1$, since $4-2\neq \pm 1$ and one is
obtained from the other by switching $2$ and $4$; the edge $\beta$
between $u=4312$ and $v=2431$ is a type two edge, since $u_4=2\neq
1,4$ and $v=u_4u_1u_2u_3$. The dotted line attached to a simplex $s$
indicates that $s$ is adjacent to some simplex $t$ in
$\Delta_{2,5}$. Since we are considering the half-open
hypersimplices, the common facet $s\cap t$ is removed from $s$.
\end{example}

\section{Proof of Theorem \ref{exc} by a shellable triangulation}
We want to show that the $h^*$-polynomial of $\Delta'_{k+1,n+1}$ is
$$\sum_{\substack{w\in\, \mathfrak{S}_{n}\\
\exc(w)=k}}t^{\des(w)}.$$
Compare this to Theorem~\ref{Shelling and Ehrhart polynomial}: if $\Delta'_{k+1,n+1}$ has a shellable unimodular triangulation $\Gamma_{k+1,n+1}$, then its $h^*$-polynomial is
$$\sum_{\alpha\in\,
\Gamma_{k+1,n+1}}t^{\#(\alpha)}.$$
We will define a shellable unimodular triangulation $\Gamma_{k+1,n+1}$ for $\Delta'_{k+1,n+1}$, label each simplex $\alpha\in \Gamma_{k+1,n+1}$ by a permutation $w_{\alpha}\in \mathfrak{S}_n$ with $\exc(w_{\alpha})=k$. Then show that $\#(\alpha)=\des(w_{\alpha})$.

We start from the triangulation $\Gamma_{k+1,n+1}$ studied in Section 3.3. By Theorem~\ref{gamma}, each simplex is labeled by a permutation $u\in \mathfrak{S}_n$ with $\des(u^{-1})=k$. Based on the Foata map defined in Section 3.2, after the following maps, the vertices of $S_{k+1,n+1}$ are permutations in $\mathfrak{S}_n$ with $k$ excedances:
\begin{equation}\label{maps}
\Gamma_{k+1,n+1}\xrightarrow{\text{rev}}R_{k+1,n+1}\xrightarrow{-1}P_{k+1,n+1}\xrightarrow{F^{-1}}
Q_{k+1,n+1}\xrightarrow{\text{rev}}S_{k+1,n+1},
\end{equation} where the map
$F^{-1}:\,P_{k+1,n+1}\rightarrow Q_{k+1,n+1}$ sending $\hat{w}$ to $w$ is the inverse of
the Foata map and the map ``rev" is the
reverse map we defined in the proof of Theorem \ref{excdes},
reversing both the letters and positions of a permutation.

\begin{example}For an
example of the above map from $\Gamma_{3,5}$ to $S_{3,5}$, consider $u=3241$. It
is in $\Gamma_{3,5}$ since $u^{-1}=4213$ has exactly two descents.
Applying the above map to $u$, we have
$$3241\xrightarrow{\text{rev}}4132\xrightarrow{-1}2431
\xrightarrow{F^{-1}}4213 \xrightarrow{\text{rev}}2431,$$
where $2431$ has $2$ excedances.
\end{example}
Apply the above maps to vertices of $\Gamma_{k+1,n+1}$, we call the new graph $S_{k+1,n+1}$. We will define the shelling order on the simplices in the triangulation by orienting each edge in the graph $S_{k+1,n+1}$. If we orient an edge $(u,v)$ such that the arrow points to $u$, then in the shelling, let the simplex labeled by $u$ be after the simplex labeled by $v$. We can orient each edge of $S_{k+1,n+1}$ (see Definition~\ref{direct}) such that the directed graph is acyclic (Corollary~\ref{acyclic}). This digraph therefore defines a partial order on the simplices of the triangulation. We will prove that any linear extension of this partial order gives a shelling order (Theorem~\ref{shelling}). Given any linear extension obtained from the digraph, the shelling number of each simplex is the number of incoming edges. Let $w_{\alpha}$ be the permutation in $S_{k+1,n+1}$ corresponding to the simplex $\alpha$. Then we can show that for each simplex, its number of incoming edges equals $\des(w_{\alpha})$ (Theorem~\ref{number}).

\begin{example}Here is the graph $S_{3,5}$ for $\Delta'_{3,5}$ with each edge oriented according to Definition~\ref{direct}.
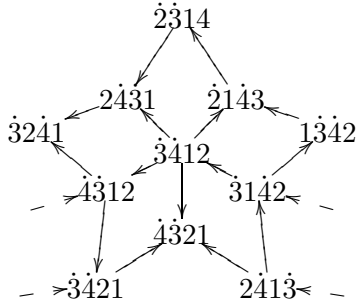
$$ \xy 0;/r.17pc/: (0,0)*{\dot{3}\dot{4}12}="a";
 (0,25)*{\dot{2}\dot{3}14}="b";
 (-10,10)*{2\dot{4}\dot{3}1}="c";
 (10,10)*{\dot{2}1\dot{4}3}="d";
 (-27,4)*{\dot{3}2\dot{4}1}="e";
 (27,4)*{1\dot{3}\dot{4}2}="f";
 (-14,-7)*{\dot{4}\dot{3}12}="g";
 (14,-7)*{\dot{3}1\dot{4}2}="h";
 (0,-15)*{\dot{4}\dot{3}21}="i";
 (-16,-25)*{\dot{3}\dot{4}21}="j";
 (16,-25)*{2\dot{4}1\dot{3}}="k";
(-28,-11)*{}="gg";
 (28,-11)*{}="hh";
 (-30,-27)*{}="jj";
 (30,-27)*{}="kk";
 {\ar "b";"c"};%
 {\ar "d";"b"};%
 {\ar "c";"e"};%
 {\ar "a";"c"};%
 {\ar "a";"d"};%
 {\ar "f";"d"};%
 {\ar "g";"e"};%
 {\ar "a";"g"};%
 {\ar "h";"a"};%
 {\ar "h";"f"};%
 {\ar "a";"i"};%
 {\ar "g";"j"};%
 {\ar "j";"i"};%
 {\ar "k";"i"};%
 {\ar "k";"h"};%
 {\ar@{-->} "jj";"j"};%
  {\ar@{-->} "hh";"h"};%
   {\ar@{-->} "kk";"k"};%
    {\ar@{-->} "gg";"g"};%
\endxy$$
For example, the vertex labeled by $3412$ with $\des(3412)=1$ has one incoming edge. Another example, consider the vertex labeled by $3142$. It has two incoming edges (including the dotted edge), which is the same as its number of descents. So we can see that it is crucial here that we are looking at the half-open hypersimplex instead of the usual hypersimplex.
\end{example}
In the following three subsections, we will first define how we orient each edge in $S_{k+1,n+1}$ and each vertex has the correct number of incoming edges, then we will show that the digraph is acyclic, and finally, any linear extension gives a shelling.
\subsection{Correct shelling number}
We need a closer look of each graph $R_{k+1,n+1},P_{k+1,n+1},Q_{k+1,n+1}$ obtained in the process of getting $S_{k+1,n+1}$ from $\Gamma_{k+1,n+1}$. First, from the description of $\Gamma_{k+1,n+1}$ (Proposition~\ref{gamma}) and  the maps in (\ref{maps}):
\begin{enumerate}
\item []$R_{k+1,n+1}$: its vertices are $u\in \mathfrak{S}_{n}$ with
$\des(u^{-1})=k$. There are two types of edges:
\begin{enumerate}
\item [1]type one edge is the
same as in $\Gamma$;
\item [2]$u$ and $v$ has a type two edge if and only if
$u_1\neq 1,n$, and $v$ is obtained from $u$ by moving $u_1$ to the
end of $u_{n}$, i.e., $v=u_2\dots u_{n}u_{1}$; or switch the role of
$u$ and $v$.
\end{enumerate}
\item []$P_{k+1,n+1}$: its vertices are $u\in
\mathfrak{S}_{n}$ with  $\des(u)=k$. There are two types of edges:
\begin{enumerate}
\item[1]
$(u,v)$ is a type one edge if and only if  the numbers $i$ and $i+1$
are not next to each other in $u$, and $v$ is obtained from $u$ by
exchanging the numbers $i$ and $i+1$. We label this edge $e_i$.
\item[2]
$(u,v)$ is a type two edge if and only if $u_1\neq 1$ and $u_n\neq
1$, and $v_i=u_{i}-1\pmod n$ for $i=1,\dots,n$ (we denote this by
$v=u-1 \pmod n$), or switch the role of $u$ and $v$.
 We label this edge $e_0$.
 \end{enumerate}
\end{enumerate}

\begin{example}Here are the graphs $R_{3,5}$ and $P_{3,5}$ for $\Delta'_{3,5}$.
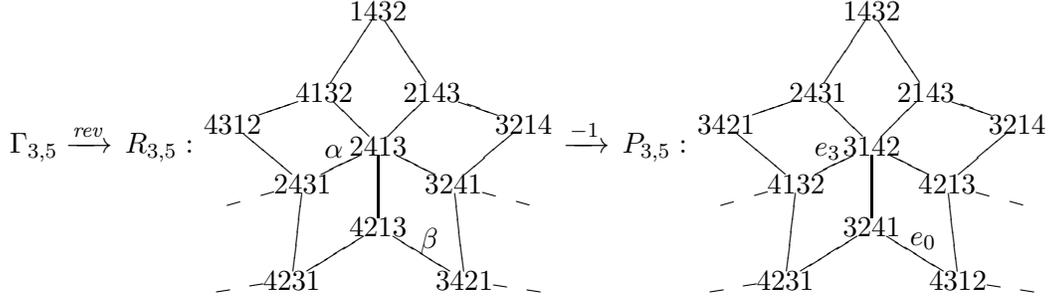
\begin{center}
$\Gamma_{3,5}\xrightarrow{\text{rev}}\,R_{3,5}:$
 $ \xy 0;/r.17pc/: (0,0)*{2413}="a";
 (0,25)*{1432}="b";
 (-10,10)*{4132}="c";
 (10,10)*{2143}="d";
 (-27,4)*{4312}="e";
 (27,4)*{3214}="f";
 (-14,-7)*{2431}="g";
 (14,-7)*{3241}="h";
 (0,-15)*{4213}="i";
 (-16,-25)*{4231}="j";
 (16,-25)*{3421}="k";
 "b"; "c"**\dir{-};
 "b"; "d"**\dir{-};
 "c"; "e"**\dir{-};
 "c"; "a"**\dir{-};
 "d"; "a"**\dir{-};
 "d"; "f"**\dir{-};
 "e"; "g"**\dir{-};
 "a"; "g"**\dir{-};
  ?*!/_-2mm/{\alpha};
 "a"; "h"**\dir{-};
 "f"; "h"**\dir{-};
 "g"; "j"**\dir{-};
 "a"; "i"**\dir{-};
 "h"; "k"**\dir{-};
 "i"; "j"**\dir{-};
 "i"; "k"**\dir{-};
   ?*!/_2mm/{\beta};
   (-28,-11)*{}="gg";
 (28,-11)*{}="hh";
 (-30,-27)*{}="jj";
 (30,-27)*{}="kk";
  "jj";"j"**\dir{--};
  "hh";"h"**\dir{--};
 "kk";"k"**\dir{--};
   "gg";"g"**\dir{--};
\endxy$ $\xrightarrow{-1}\,P_{3,5}:$
 $ \xy 0;/r.17pc/: (0,0)*{3142}="a";
 (0,25)*{1432}="b";
 (-10,10)*{2431}="c";
 (10,10)*{2143}="d";
 (-27,4)*{3421}="e";
 (27,4)*{3214}="f";
 (-14,-7)*{4132}="g";
 (14,-7)*{4213}="h";
 (0,-15)*{3241}="i";
 (-16,-25)*{4231}="j";
 (16,-25)*{4312}="k";
 "b"; "c"**\dir{-};
 "b"; "d"**\dir{-};
 "c"; "e"**\dir{-};
 "c"; "a"**\dir{-};
 "d"; "a"**\dir{-};
 "d"; "f"**\dir{-};
 "e"; "g"**\dir{-};
 "a"; "g"**\dir{-};
  ?*!/_-2mm/{e_3};
 "a"; "h"**\dir{-};
 "f"; "h"**\dir{-};
 "g"; "j"**\dir{-};
 "a"; "i"**\dir{-};
 "h"; "k"**\dir{-};
 "i"; "j"**\dir{-};
 "i"; "k"**\dir{-};
   ?*!/_2mm/{e_0};
   (-28,-11)*{}="gg";
 (28,-11)*{}="hh";
 (-30,-27)*{}="jj";
 (30,-27)*{}="kk";
  "jj";"j"**\dir{--};
  "hh";"h"**\dir{--};
 "kk";"k"**\dir{--};
   "gg";"g"**\dir{--};
\endxy$
\end{center}
In the graph $R_{3,5}$ above, the edge labeled $\alpha$ is of type one
switching $1$ and $3$; and $\beta$ is of type two, with $u=3421$ and
$v=u_2u_3u_4u_1=4213$. In the above graph $P_{3,5}$, the edge $e_3$ is an
edge of type one between $u=4132$ and $3142$ switching $3$ and $4$
since they are not next to each other; and the edge $e_1$ between
$u=4312$ and $v=3241=u-1\pmod 4$ is of type two.
\end{example}
\begin{definition}Let $w\in \mathfrak{S}_n$. Define its \emph{descent set}
to be $\Des(w)=\{i\in[n-1]\mid w_i>w_{i+1}\}$  its \emph{leading
descent set} to be the actual numbers on these positions,
$\LdDes(w)=\{w_{i}\mid i\in\Des(w)\}$.
\end{definition}
For $w\in P_{k+1,n+1}$, since $\des(w)=k$, we have
$\Des(w)=\{i_1,\dots,i_k\}$. By the description of edges in $P_{k+1,n+1}$, we
have the following relation of $\Des$ and $\LdDes$ for an edge in
$P_{k+1,n+1}$:
\begin{lemma}\label{desld}Let $v$ be a vertex in $P_{k+1,n+1}$.
\begin{enumerate}
\item Define $u$ by $v=u-1 \pmod n$. There are three cases depending on the position of the letter $n$ in $v$:
\begin{enumerate}
\item if $v_1=n$, then $\Des(u)=\Des(v)\backslash\{1\}$, thus $u\in P_{k,n+1}$;
\item if $v_n=n$, then $\Des(u)=\Des(v)\cup\{n-1\}$ and $u\in P_{k+2,n+1}$;
\item if $v_i=n$ with $i\neq 1,n$, then
$\Des(u)=\Des(v)\cup\{i-1\}\backslash\{i\}$ and $u\in P_{k+1,n+1}$.
\end{enumerate}

\item Let $e_i=(u,v)$ be a type one edge in $P_{k+1,n+1}$. Then
 we have $\Des(u)=\Des(v)$. In this case, we also
compare $\LdDes(u)$ and $\LdDes(v)$:
\begin{enumerate}
\item if $i,i+1\in \LdDes(u)$ or $i,i+1\notin \LdDes(u)$, we have
$\LdDes(u)=\LdDes(v)$;
\item otherwise, if  $i\in \LdDes(u)$ and $i+1\notin \LdDes(u)$, we
have $\LdDes(v)=\LdDes(u)\cup\{i+1\}\backslash\{i\}$.
\end{enumerate}
\end{enumerate}
\end{lemma}

 Now consider the map from $P_{k+1,n+1}$ to $S_{k+1,n+1}$. Notice that this map is the same as defined in Theorem \ref{excdes}.
Therefore, we have
\begin{corollary}\label{QS}Vertices in $Q_{k+1,n+1}$ are permutations $w\in \mathfrak{S}_{n}$ with
 $k$ reverse excedances ($i$ such that $w_i<i$), and vertices in $S_{k+1,n+1}$ are
permutation $v\in \mathfrak{S}_{n}$ with $\exc(v)=k$. Moreover,
 the \emph{reverse excedances set} in $w\in Q_{k+1,n+1}$, denoted by
$\Rexc(w)=\{i\mid w_i<i\}$ is the same as $\LdDes(\hat{w})$, where
$\hat{w}=F(w)\in P_{k+1,n+1}$. So part 2 of Lemma \ref{desld} for
$\LdDes(\hat{w})$ also apply for $\Rexc(w)$.
\end{corollary}

For $w\in Q$,  decompose $[n-1]$ by $A_w\cup B_w\cup C_w$ (disjoint
union), where
\begin{align}
A_w=&\{i\in [n-1]\mid i\notin\Rexc(w),\,i+1\in\Rexc(w)\},\\
B_w=&\{i\in [n-1]\mid i+1\notin\Rexc(w),\,i\in\Rexc(w)\}, \text{and}\\
\label{C}C_w=&\{i\in [n-1]\mid i,i+1\notin\Rexc(w)\text{ or
}i,i+1\in\Rexc(w)\}.
\end{align}
For example, consider $v=54\dot{1}\dot{2}6\dot{3}8\dot{7}9$, where
the dotted positions are in $\Rexc(v)$. Then $A_v=\{2,5,7\}$,
$B_v=\{4,8,6\}$ and $C_v=\{1,3\}$.

 For an edge $(u,v)\in Q_{k+1,n+1}$, we label it $e_i$ according to the
 labeling
of the corresponding edge $e_i=(\hat{u},\hat{v})\in P_{k+1,n+1}$. Then we
orient each edge in $Q_{k+1,n+1}$ in the following way:
\begin{definition}\label{direct}
Let $e_i=(u,v)$ be an edge in $Q_{k+1,n+1}$.
\begin{enumerate}
\item For type one edge ($i\neq 0$),
\begin{enumerate}
\item if $\Rexc(u)\neq\Rexc(v)$, then define $u\rightarrow
v$ if and only if $i\in \Rexc(v)$ (this implies $i\notin \Rexc(u)$
by Lemma \ref{desld} part 2(b));
\item if $\Rexc(u)=\Rexc(v)$, then define $u\rightarrow v$ if and only if $v_i>v_{i+1}$ (this implies
$u_i<u_{i+1}$ by Corollary \ref{weak}).
\end{enumerate}
\item For type two edge ($i=0$), define $u\rightarrow
v$ if and only if $\hat{v}=\hat{u}-1 \pmod n$, where
$(\hat{u},\hat{v})$ is the corresponding edge in $P_{k+1,n+1}$.
\end{enumerate}
\end{definition}
Based on the above definition and the Foata map, we have the following
description of incoming edges for $v\in Q_{k+1,n+1}$.
\begin{lemma}\label{edgeQ}Let $v$ be a vertex in $Q_{k+1,n+1}$. Then
\begin{enumerate}
\item $v$ has an incoming type one edge ($e_i$ with $i\neq 0$) if
and only if one of the following two holds:
\begin{enumerate}
\item $i\in B_v$;
\item $i\in C_v\cap\Des(v)$.
\end{enumerate}
\item  $v$ has an incoming type two edge ($e_0$) if
and only if $v_n\neq n$.
\end{enumerate}
\end{lemma}
\begin{proof}\begin{enumerate}
\item First, by Definition \ref{direct}, and Lemma \ref{desld}, it is clear that if there
exists an edge $e_i$ with $i\neq 0$, $u\rightarrow v$ for some $u\in
Q_{k+1,n+1}$, then $v$ satisfies one of conditions (a) and (b). On the other
hand, we need to show that, if  (a) or (b) is true for $v$, then
there exists an edge $e_i=(u,v)\in Q_{k+1,n+1}$. Then, by Definition
\ref{direct}, the edge will be oriented as $u\rightarrow v$. In
fact, consider the corresponding permutation $\hat{v}\in P_{k+1,n+1}$. From
the description of $P_{k+1,n+1}$, $\hat{v}$ has a type one edge $e_i$ if and
only if $i$ and $i+1$ are not next to each other in $\hat{v}$. But
with a careful look at the inverse Foata map, we can see that if (a)
or (b) is true for $v$, then neither the case $\hat{v}=\dots
i(i+1)\dots$ nor $\hat{v}=\dots (i+1)i\dots$ can be true.

\item Let $\hat{u}=\hat{v}+1 \pmod n$ in $P_{k+1,n+1}$. If $\des(\hat{u})=\des(\hat{v})$,
then we have $\hat{u}\in P_{k+1,n+1}$, so $\hat{v}$ has a type two edge, and
this edge points to $v$ by Definition \ref{direct}. If
$\des(\hat{u})=\des(\hat{v})-1$, then $\hat{v}$ still has an
incoming edge $e_1$, since we are considering the half-open
hypersimplex and this edge indicates that the common facet $u\cap v$
is removed from $v$. Then by Lemma \ref{desld}, part 1, $\des(\hat{u})\le
\des(\hat{v})$ if and only if case (b) does not happen, i.e.,
$\hat{v}_n\neq n$ in $P$. This is equivalent to $v_n\neq n$ in $Q_{k+1,n+1}$ by the inverse Foata map.\qedhere
\end{enumerate}
\end{proof}
\begin{definition}\label{bigsmall}Let $I,J\subset[n-1]$. Define a \emph{big
block} of $Q_{k+1,n+1}$ to be $b_I=\{w\in Q_{k+1,n+1}\mid \Des(\hat{w})=I\}$, where
$\hat{w}=F(w)\in P_{k+1,n+1}$. Define a \emph{small block} $s_{I,J}=\{w\in
b_I\mid \Rexc(w)=J\}$. We say the small block $s_{I,J}$ is smaller
than $s_{I',J'}$ if 1) $I<I'$ or 2) $I=I'$ and $J>J'$.
\end{definition}
For two different sets $I,I'\subset[n-1]$ with $I=\{i_1\le\dots \le
i_k\}$ and $I'=\{i'_1\le\dots \le i'_{\ell}\}$, we define $I<I'$ if
1) $k<\ell$ or 2) $k=\ell$ and $i_j\le i'_j$ for all $j=1,\dots,k$.
Then by Lemma \ref{desld} and Lemma \ref{edgeQ}, we have
\begin{corollary}\label{edgechange}For an edge $e_i=u\rightarrow v\in Q_{k+1,n+1}$ with $u\in
s_{I,J}$ and $v\in s_{I',J'}$,
\begin{enumerate}
\item if $i=0$, then $I'>I$;
\item if $i\neq 0$ and $i\in B_v$, then $I=I'$ and $J'<J$;
\item if $i\neq 0$ and $i\in C_v$, then $I=I'$ and $J=J'$.
\end{enumerate}
\end{corollary}
\begin{example}\label{PQ}Here is an example of Definition \ref{direct}, Lemma \ref{edgeQ}
 and Corollary \ref{edgechange}, with a
type one edge drawn in $Q_{3,5}$ and a type two ($e_0$) in $P_{3,5}$ for
$\Delta'_{3,5}$.
 \begin{center} $P_{3,5}:$
 $ 
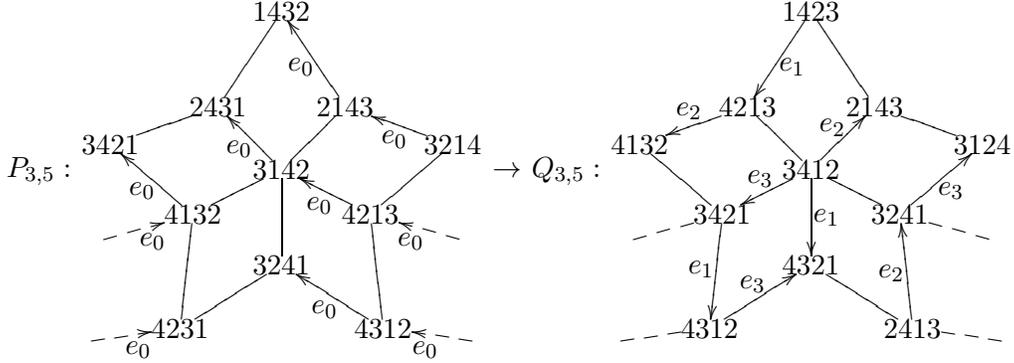
\xy 0;/r.2pc/: (0,0)*{3142}="a";
 (0,25)*{1432}="b";
 (-10,10)*{2431}="c";
 (10,10)*{2143}="d";
 (-27,4)*{3421}="e";
 (27,4)*{3214}="f";
 (-14,-7)*{4132}="g";
 (14,-7)*{4213}="h";
 (0,-15)*{3241}="i";
 (-16,-25)*{4231}="j";
 (16,-25)*{4312}="k";
 "b"; "c"**\dir{-};
{\ar "d";"b"};%
?*!/_2mm/{e_0}; "c"; "e"**\dir{-};
 {\ar "a";"c"};%
 ?*!/_2mm/{e_0};
 "d"; "a"**\dir{-};
{\ar "f";"d"};%
?*!/_2mm/{e_0};
 {\ar "g";"e"};%
 ?*!/_2mm/{e_0};
 "a"; "g"**\dir{-};
 {\ar "h";"a"};%
 ?*!/_2mm/{e_0};
 "f"; "h"**\dir{-};
 "g"; "j"**\dir{-};
 "a"; "i"**\dir{-};
 "h"; "k"**\dir{-};
 "i"; "j"**\dir{-};
  {\ar "k";"i"};%
  ?*!/_2mm/{e_0};
   (-28,-11)*{}="gg";
 (28,-11)*{}="hh";
 (-30,-27)*{}="jj";
 (30,-27)*{}="kk";
 {\ar@{-->} "jj";"j"};%
 ?*!/_-2mm/{e_0};
  {\ar@{-->} "hh";"h"};%
  ?*!/_2mm/{e_0};
   {\ar@{-->} "kk";"k"};%
   ?*!/_2mm/{e_0};
    {\ar@{-->} "gg";"g"};%
    ?*!/_-2mm/{e_0};
\endxy$ $\rightarrow$ $Q_{3,5}:$
$ \xy 0;/r.2pc/: (0,0)*{3412}="a";
 (0,25)*{1423}="b";
 (-10,10)*{4213}="c";
 (10,10)*{2143}="d";
 (-27,4)*{4132}="e";
 (27,4)*{3124}="f";
 (-14,-7)*{3421}="g";
 (14,-7)*{3241}="h";
 (0,-15)*{4321}="i";
 (-16,-25)*{4312}="j";
 (16,-25)*{2413}="k";
 {\ar "b";"c"};%
   ?*!/_2mm/{e_1};
 "b"; "d"**\dir{-};
  {\ar "c";"e"};%
  ?*!/_-2mm/{e_2};
 "c"; "a"**\dir{-};
  {\ar "a";"d"};%
 ?*!/_2mm/{e_2};
 "d"; "f"**\dir{-};
 "e"; "g"**\dir{-};
 "a"; "h"**\dir{-};
  {\ar "h";"f"};%
   ?*!/_-2mm/{e_3};
  {\ar "k";"h"};%
  ?*!/_2mm/{e_2};
 "i"; "k"**\dir{-};
 {\ar "a";"g"};%
    ?*!/_-2mm/{e_3};
 {\ar "a";"i"};%
  ?*!/_2mm/{e_1};
 {\ar "g";"j"};%
  ?*!/_-2mm/{e_1};
 {\ar "j";"i"};%
  ?*!/_2mm/{e_3};
(-28,-11)*{}="gg";
 (28,-11)*{}="hh";
 (-30,-27)*{}="jj";
 (30,-27)*{}="kk";
  "jj";"j"**\dir{--};
  "hh";"h"**\dir{--};
 "kk";"k"**\dir{--};
   "gg";"g"**\dir{--};
\endxy$
\end{center}
It is clear from the graph $P_{3,5}$ that $\hat{v}\in P_{3,5}$ has an incoming
$e_0$ if and only if $\hat{v}_4\neq 4$, which is equivalent to
$v_4\neq 4$ in $Q_{3,5}$. Consider $v=4321\in Q_{3,5}$. It has
$\Rexc(v)=\{3,4\}$. Since $v_4\neq 4$, it has an incoming $e_0$ edge
(shown in $P_{3,5}$); since $v$ with $i=1,3$ satisfies condition b) in
Lemma \ref{edgeQ}, there are two incoming edges $e_1$ and $e_3$ of
type two, and these are all the incoming edges of $v$.

Consider the edge $e_0=u\rightarrow v\in Q_{3,5}$ whose corresponding edge
in $P_{3,5}$ is between $\hat{u}=4312$ and $\hat{v}=3241$. We have
$I=\Des(\hat{u})=\{1,2\}$ and $I'=\Des(\hat{v})=\{1,3\}$, with
$I'>I$. Consider the edge $e_2=u\rightarrow v\in Q_{3,5}$ with
$u=34\dot{1}\dot{2}$ and $v=2\dot{1}4\dot{3}$, where the dotted
positions are in $\Rexc$. Since $2\in B_v$, we have
$J=\Rexc(u)=\{3,4\}$ and $J'=\Rexc(v)=\{2,4\}$ with $J>J'$. Finally,
consider $e_1=u\rightarrow v\in Q_{3,5}$ with $u=34\dot{2}\dot{1}$ and
$v=43\dot{1}\dot{2}$. Since $1\in C_v$, we have
$\Rexc(u)=\{3,4\}=\Rexc(v)$.
\end{example}

With the orientation of $Q_{k+1,n+1}$ by Definition \ref{direct}, we have
\begin{theorem}\label{number}
For each vertex $v\in Q_{k+1,n+1}$, the number of its incoming edges equals
$\des(v)$.
\end{theorem}

\begin{proof}First, notice that if $i\in B_v$,
then $i\notin\Des(v)$; and if $j\in A_v$, then $j\in \Des(v)$. So
$\Des(v)=A_v\cup (C_v\cap \Des(v))$.  Now we will define a bijection
between the set $\Des(v)$ and the set of incoming edges of $v$ as
listed in Lemma \ref{edgeQ}. First notice that $i\in \Des(v)\cap
C_v$ corresponds to an incoming edge $e_i$ described in  case $(b)$
of Lemma \ref{edgeQ}. Then we need to match $A_v$ with the set of
incoming edges in Lemma \ref{edgeQ}, part 1 (a) and part 2. There
are two cases:
\begin{enumerate}
\item If $v_{n}= n$, by Lemma \ref{edgeQ}, $v$ does not have a type two incoming edge.
Then we have a bijection between the sets $A_v$ and $B_v$ by
matching $i\in A$ to $\min\{j\in B\mid j>i\}$. For example,
$v=54\dot{1}\dot{2}6\dot{3}8\dot{7}9$ where the dotted positions are
in $\Rexc(v)$. Then $A_v=\{2,5,7\}$ is in bijection with
$B_v=\{4,8,6\}$. This gives us the desired bijection since the set
of $i$'s such that $e_i$  is a type one incoming edge of case 1 (a)
is exactly $B_v$.

\item If $v_{n}\neq n$, $A_v$ has one element more than $B_v$, since the largest number
in $A_v$ does not have image in $B_v$. But since in this case, $v$
has a type two incoming edge by Lemma \ref{edgeQ}, the extra descent
can be taken care by this incoming edge.\qedhere
\end{enumerate}

\end{proof}

\subsection{Acyclicity}
We want to show that the digraph defined in the previous subsection gives a shelling order. First, we need to show that any linear extension of the above ordering is well defined, i.e., there is no cycle in the directed graph $S_{k+1,n+1}$ (equivalently, $Q_{k+1,n+1}$ is acyclic). In this section, we restrict to the connected component of a small
block of $Q_{k+1,n+1}$, i.e., the subgraph of $Q_{k+1,n+1}$ consisting of permutations
with the same $\Rexc$, or equivalently, the subgraph of $P_{k+1,n+1}$
consisting of permutations with the same $\LdDes$. By Lemma
\ref{desld} and Lemma \ref{edgeQ}, $e_i=(u,v)\in Q_{k+1,n+1}$ with $u$ and $v$
in the same small block if and only $i\in C_u$, where $C_u$ is
defined in (\ref{C}). We want to show that there is no directed
cycle in each small block of $Q_{k+1,n+1}$.

For a permutation $w$, let $t_i(w)$ be the permutation obtained by
switching letters $i$ and $i+1$ in $w$, and $s_i(w)$ be the
permutation obtained by switching letters in positions $i$ and
$i+1$. Now consider $e_i=(u,v)\in Q_{k+1,n+1}$ and the corresponding edge
$(\hat{u},\hat{v})\in P$. By definition of $P_{k+1,n+1}$, we have
$\hat{u}=t_i\hat{v}$. Then in $Q_{k+1,n+1}$, we have:

\begin{lemma}\label{single}Let  $e_i=(u,v)\in Q_{k+1,n+1}$ and $i\in C_u$. Then
$$u=\begin{cases}
s_i(v), & \CT(u)\neq \CT(v)\\
t_is_i(v), & \CT(u)= \CT(v)
\end{cases},
$$
where $\CT(w)$ stands for the cycle type of $w$ defined in Section
1.3.
\end{lemma}
\begin{proof}In $P_{k+1,n+1}$, we have $\hat{u}=\dots i\dots (i+1)\dots$ and  $\hat{v}=\dots (i+1)\dots
i\dots$. By the inverse Foata map and  the condition that $u$ and
$v$ are in the same small block, i.e., $\LdDes(u)=\LdDes(v)$, we can
see that the only case when $\CT(u)\neq \CT(v)$ is
$u=\dots(i\dots)((i+1)\dots)\dots$ and $v=\dots((i+1)\dots
i\dots)\dots$ (in standard cycle notation). Then the conclusion
follows from the inverse Foata map.
\end{proof}
\begin{example}Consider $u=4321\in Q_{3,5}$ with standard cycle notation
 $u=(32)(41)$ in Example \ref{PQ}. For $e_3=(u,v)$ with $v=4312=(4231)$,
 since $\CT(u)\neq\CT(v)$, we have $u=s_3(v)$. For $e_1=(u,v')$ with $v'=3412=(31)(42)$,
 since $\CT(u)=\CT(v)$, we have $u=t_3s_3(v)$, i.e., $u$ is obtained from $v'$ by switching
 $3$ and $4$, which is $4312$, and then switching $v'_3$ and $v'_4$,
 which is $4321=u$.
\end{example}
For a permutation $w$, define its \emph{inversion set} to be
$\inv(w)=\{(w_i,w_j)\mid i<j,\,w_i>w_j\}$ and denote $\#\inv(w)$ by
$i(w)$. By Lemma \ref{single}, we have
\begin{corollary}\label{weak}For $e_i=(u,v)\in Q_{k+1,n+1}$  with $i\in\Des(u)$ and $i\in C_u$,
 we have $i\notin \Des(v)$ and $i(v)\le i(u)$.
\end{corollary}

  Now  consider
a sequence of edges $E$ in a small block of $Q_{k+1,n+1}$:
$u\leftarrow\cdots\leftarrow v$. By Corollary \ref{weak}, we have
$i(v)\le i(u)$. In order to show that there is no cycle in each
small block, we find an invariant that strictly decreases along any
directed path. We define the $E$-\emph{inversion set} to be
\begin{equation}\label{inve}
\inv_E(w)=\{(w_i,w_j)\in \inv(w)\mid \{e_i,\dots,e_{j-1}\}\subset
E\}
\end{equation}
 and claim that $i_E(w)=\#\inv_E(w)$
 is such an invariance (Lemma \ref{strict}).
\begin{example}For $w=361452798$ and $E=\{e_i\mid i\in\{2,3,5\}\}$, we first cut $w$ into blocks (indicated by lines):
$$w=\underline{3}\,\underline{614}\,\,\underline{52}\,\,\underline{7}\,\,\underline{9}\,\,\underline{8},$$
with the property that each block can be permutated arbitrarily by
$\{s_i\mid i\in \{2,3,5\}\}$. Then
$\inv_E(w)=\{(6,1),(6,4),(5,2)\}$, i.e., $(w_i,w_j)\in \inv(w)$ with
$w_i,w_j$ in the same block. This is the same as in (\ref{inve}).

Here are three extremal examples. If $E=\{e_i\mid i\in[n-1]\}$, then
$\inv_E(w)=\inv(w)$ for all $w\in \mathfrak{S}_n$. If $E=\{e_i\}$
and $e_i=u\leftarrow v\in Q$, then  $i_E(u)=1$ and $i_E(v)=0$. If
$E=\{e_i\mid i\in I\subset [n-1]\}$ and $i\notin \Des(w)$ for all
$i\in I$, then $i_E(w)=0$. This is the situation in Lemma
\ref{strict}.
\end{example}

\begin{lemma} \label{strict}
Let $u\leftarrow\cdots\leftarrow v$ be a sequence of edges in a
small block of $Q_{k+1,n+1}$ with edge set $E$. Then $i_E(v)<i_E(u)$.
\end{lemma}
\begin{proof}By Lemma \ref{single} and Lemma \ref{edgeQ}, Part 1,
 we have
 $i_E(v)\le i_E(u)$.
Suppose we have $i_E(v)= i_E(u)$. We will show that no edge can
belong to $E$.
First, we show that $e_{n-1}$ cannot be in $E$. Let $w$ be any
permutation in the above  path from $v$ to $u$. If $n-1\notin C_w$,
then certainly $e_{n-1}\notin E$. Suppose $n-1\in C_w$. Let
$\hat{w}'=t_{n-1}(\hat{w})$ in $P$. Notice that we always have
$\CT(w')\neq\CT(w)$. Then by Lemma \ref{single}, we have
$w'=s_i(w)$, and thus $i_E(w')<i_E(w)$. So $e_{n-1}\notin E$.

Now consider $u$. Notice that the last cycle of $u$ in its standard
cycle notation must start with $n$. Let the cycle be
$(n\,a_1\,a_2\cdots a_k)$. We claim that $e_{a_1}\notin E$. First,
for $u$, since $u_n=a_1$ and $e_{n-1}\notin E$, all pairs
$(u_i,u_n)$ are not in $\inv_E(u)$ by definition of $\inv_E$ in
(\ref{inve}). Let $\hat{u}'=t_{a_1}(\hat{u})$ in $P$. Independent of
the fact that $\CT(u)=\CT(u')$, we have $i_E(u')<i_E(u)$. Now
consider any $w$ appearing in the path from $v$ to $u$. Suppose all
edges before $w$ are not $e_{a_1}$. Then
 we still have $w_n=a_1$. Let $\hat{w}'=t_{a_1}(\hat{w})$. Then again consider
 both cases $\CT(w)=\CT(w')$ or not, by Lemma \ref{single},
 we have $i_E(w')<i_E(w)$. Therefore, $e_{a_1}\notin E$.

With the same argument, we can show that $e_{a_2}\notin E$, $\dots$,
$e_{a_k}\notin E$. Then we can move to the previous cycle, until we
have $e_i\notin E$ for all $i\in[n-1]$.
\end{proof}
\begin{corollary}\label{acyclic}$Q_{k+1,n+1}$ is acyclic.
\end{corollary}

\begin{proof}
 First it is not hard to see that there is no cycle that involves vertices
 in different small blocks, since both big blocks and
small blocks have the structure of a poset, and edges between two
small/big blocks all have the same direction. Therefore, if
$Q_{k+1,n+1}$ has a cycle, it has to be within a small block.

Suppose there is a directed cycle within a small block with edge set
$E$. Consider some $w$ in the cycle and let $w_1=w_2=w$ in Lemma
\ref{strict}, we will have $i_E(w)<i_E(w)$, a contradiction.
\end{proof}
\subsection{Shellable triangulation}
In this section, we will show that any linear extension of the
ordering of the simplices in $Q_{k+1,n+1}$ is shellable. We will prove this by
showing that each simplex has a unique minimal nonface (see Section
3.1).

 Let us first assign a face $F$ to each simplex. Each
incoming edge $\alpha\xleftarrow{e_i} \alpha_i$ defines a unique
vertex $M_i$ of $\alpha$ that $\alpha $ has but $\alpha_i$ does not
have. Then let $F=\{M_i\}$ be given by all the incoming edges of
$\alpha$. We want to show that $F$ is the unique minimal nonface of
$\alpha$. First, let us assume $F$ is a nonface, i.e., it has not
appeared before $\alpha$ in a given order of simplices. We can see
that $F$ is the unique minimal nonface, i.e., any proper subface of
$F$ has appeared before. In fact, let $M_i$ be a vertex in $F$ but
not in $F'\subset F$. Then we have $F'\subset \alpha_i$ since
$\alpha_i$ has every vertex of $\alpha$ except for $M_i$.

In the rest of this section, we will show that $F$ is a nonface. To
show this, let $Q_F$ be the (connected) component of $Q_{k+1,n+1}$ consisting
of all simplices containing $F$. Then it suffices to show that
$\alpha$ is the only source of $Q_F$, and any other simplices are
reachable by $\alpha$, i.e., there exists a directed path from
$\alpha$ to that simplex. We will first show this within each small
block and then connect different small blocks.

Let $Q_{F,s}$ be a (connected) component of $Q_{k+1,n+1}$ consisting of all
simplices in a small block $s$ containing $F$. In Section 5, we
define a ``vertex expression" for each simplex in $\Delta$. Let the
vertex expression of two simplices be $\alpha=M_1\dots M_{n+1}$ and
$\beta=M'_1 \dots  M'_{n+1}$. Assume $\alpha$ and $\beta$ are
connected by an edge $e_i$. Then by Corollary \ref{position},
$\alpha$ and $\beta$ differs only by the $(i+1)$th vertex, i.e.,
$M_{i+1}\neq M'_{i+1}$ and $M_j=M'_j$ for all $j\neq i+1$. Then it
follows  that there exists an edge set $E$ for $Q_{F,s}$, such that
$Q_{F,s}$ is closed under this edge set: if $\beta$ is connected to
$\alpha$ by an edge $e\in E$ and $\alpha\in Q_{F,s}$, then $\beta\in
Q_{F,s}$. In fact, let $\alpha\in Q_{F,s}$ and say the vertices of
$F$ are in the positions $J\subset [n]$ of $\alpha$. Then we have
$E=\{e_i\mid i\notin J, \,i\in C_u, \text{ for any }u\in s\}$. To
show the nonface property for each small block (Corollary
\ref{smallshelling}), we need the following lemma about the Foata
map.

\begin{lemma}\label{pureperm}
 Let $I\subset\{1,2,\dots,n-1\}$. For a permutation $w\in\mathfrak{S}_n$, consider the set $E(w)$ of all the permutations
obtained by applying any sequence of $t_i$, $(i\in I)$ to $w$, i.e.,
$$E(w)=\{u=t_{i_1}\dots t_{i_k}(w)\mid i_j\in I \text{ for some }k\}.$$ Then there exists a unique $u\in E(w)$ such that $F^{-1}(u)$
has ascents in $I$.
\end{lemma}

\begin{proof}
We can describe an algorithm to determine this $u$ uniquely. First,
notice that the group generated by $t_i$, $(i\in I)$ is a subset of
the symmetric group $S_n$, and has the form $S_{a_1}\times
S_{a_2}\times \cdots\times S_{a_k}$, where $a=(a_1,a_2,\dots,a_k)$
is a composition of $n$. For example, if $n=9$, and
$I=\{2,3,5,7,8\}$, then $a=(1,3,2,3)$. A composition in $k$ parts
divides the numbers $1,2,\dots,n$ into $k$ parts, and numbers in
each region can be permuted freely by $t_i$, $(i\in I)$.

Now in the given $w$, replace numbers in each region by a letter and
order the letters by the linear order of the regions. In the
previous example, replace $\{1\}$, $\{2,3,4\}$, $\{5,6\}$ and
$\{7,8,9\}$ by $a,b,c,d$ respectively and we have the order
$a<b<c<d$. For example, if $w=253496187$, then we get a word
$bcbbdcadd$.

Next, add parentheses to the word in front of each left-to-right
maximum, as in the inverse Foata map. For $bcbbdcadd$, we have
$(b)(cbb)(dcadd)$. Notice that we do not have parentheses before the
second and third $d$. No matter how we standardize this word, the
cycles we get will be a refinement of the cycles for the word.

Now comes the most important part. We want to standardize the word
in a way such that $v=F^{-1}(w)$ is increasing in all positions of
$I$. To do this, we look at a letter in the word and compare it to the
next word it goes to in the cycle notation. For example, consider
the $b$'s in $(b_1)(cb_2b_3)(dcadd)$.  $v_{b_1}=b_1\in\{2,3,4\}$,
$v_{b_2}=b_3\in\{2,3,4\}$ and $v_{b_3}\in\{5,6\}$. Since
$v_{b_3}>v_{b_1}$ and $v_{b_2}$, to keep $v$ increasing in positions
$\{2,3,4\}$, we have $b_3>b_1$ and $b_2$, so $b_3=4$. Now continue
to compare  $b_1$ and $b_2$. Since $v_{b_2}=b_3>v_{b_1}=b_1$, we
have $b_1<b_2$, and thus $b_1=2$, $b_2=3$. Notice that if there are
no periodic cycles, then we can always choose a unique way to
standardize the letters to a permutation with the required property.
For a periodic cycle, there is still a unique way to standardize
them, which is to standardize each letter in the cycle increasingly.
For example, for $(baba)$, $(3142)$ is the unique way. This
completes the algorithm and proof.
\end{proof}

\begin{corollary}[small block shelling]\label{smallshelling}
   For any face $F\subset \Delta'$, if $Q_{F,s}\neq\emptyset$, then $Q_{F,s}$ has only one source and any other simplices are reachable by that source.
\end{corollary}
\begin{proof}
Let $E$ be the edge set corresponding to $Q_{F,s}$. By Lemma
\ref{edgeQ} part 1 (b), if $\alpha$ is a source in $Q_{F,s}$, then
$i_{E}(\alpha)=0$. First, by Lemma \ref{strict}, we know that there
exists at least one  such source. In fact, let $\alpha\in Q_{F,s}$.
If $i_{E}(\alpha)\neq0$, then by Lemma \ref{edgeQ} part 1 (b), we
can keep going along the incoming edges of $Q_{F,s}$. And since
there is no cycle within the small block and there are only finitely
many simplices in $Q_{F,s}$, we will reach a source.

Now by Lemma \ref{pureperm}, there is at most one source for
$Q_{F,s}$. Then the proposition is proved since the above ``tracing
back along arrows'' will guarantee that each simplex in $Q_{F,s}$ is
reachable by that unique source.
\end{proof}
\begin{theorem} \label{shelling}Any linear extension of the above defined
ordering between adjacent simplices will give a shelling order for
the half-open hypersimplex.
\end{theorem}

\begin{proof} It suffices to show that for each
face $F$ in $\Delta'$,  $Q_F$ has only one source and any other
simplices are reachable by that source. First by Proposition
\ref{onesmall},  $Q_F$ starts with a unique minimal connected small
block. By Lemma \ref{edgeQ}, 1(b) and 2, each simplex in $\Delta'$
has an incoming edge from a simplex in a smaller small block.
Therefore, the source $\alpha_F$ in the unique minimal small block
of $Q_F$ is the unique source of $Q_F$, and each simplex in $Q_F$ is
reachable from $\alpha_F$ via the unique source in each $Q_{F,s}$.
\end{proof}

\section{Vertex expression for simplices in the triangulation}
Let $z_i=x_1+\dots+x_i$, we have an equivalent definition for
$\Delta_{k+1,n+1}$:
$$\Delta_{k+1,n+1}=\{(z_1,\dots,z_{n})\mid 0\le z_1,z_2-z_1,\dots,z_{n}-z_{n-1}\le
1;\,k\le z_{n}\le k+1\},$$
In this new coordinate system, the
  triangulation of $\Delta_{k+1,n+1}$ is called the alcoved triangulation  \cite{lp}.

 Now
all the integral points will be vertices of some simplex in the
triangulation. Denote the set of all the integral points in
$\Delta_{k+1,n+1}$ by $V_{k+1,n+1}=\{\Z^n\cap \Delta_{k+1,n+1}\}$ .
Now we define a partial order on   $V_{k+1,n+1}$ (we will drop the
indices from now on). For $M=(m_1,\dots,m_n),N=(m'_1,\dots,m'_n)\in
V$, we define $M>N$ if and only if $m_i\ge m'_i$ for $i=1,\dots,n$.
If $M=N+e_i$, where $e_i$ is the vector with 1 in the $i$th position
and 0 elsewhere, then label this edge in the Hasse diagram by
$n+1-i$. We still call the Hasse diagram of this poset on $V_{k+1,n+1}$ by $V_{k+1,n+1}$
itself. Here is an example of $V_{3,5}$.

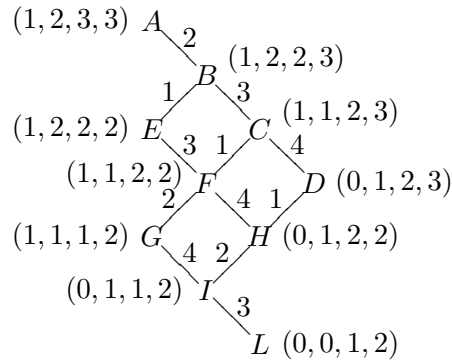
\begin{center}
 $ \xy 0;/r.17pc/:
 (-10,30)*{A}="a"; (-25,30)*{(1,2,3,3)};
 (0,20)*{B}="b";(15,23)*{(1,2,2,3)};
 (10,10)*{C}="c";(25,13)*{(1,1,2,3)};
  (20,0)*{D}="d";(35,0)*{(0,1,2,3)};
 (0,0)*{F}="f";(-15,1.5)*{(1,1,2,2)};
 (-10,10)*{E}="e"; (-25,10)*{(1,2,2,2)};
  (-10,-10)*{G}="g";(-25,-10)*{(1,1,1,2)};
   (10,-10)*{H}="h";(25,-10)*{(0,1,2,2)};
    (0,-20)*{I}="i";(-15,-20)*{(0,1,1,2)};
     (10,-30)*{L}="l";(25,-30)*{(0,0,1,2)};
 "a"; "b"**\dir{-};
  ?*!/_2mm/{2};
 "b"; "c"**\dir{-};
  ?*!/_2mm/{3};
 "c"; "d"**\dir{-};
  ?*!/_2mm/{4};
 "e"; "f"**\dir{-};
  ?*!/_2mm/{3};
 "f"; "h"**\dir{-};
  ?*!/_2mm/{4};
 "g"; "i"**\dir{-};
  ?*!/_2mm/{4};
 "i"; "l"**\dir{-};
  ?*!/_2mm/{3};
 "e"; "b"**\dir{-};
  ?*!/_2mm/{1};
 "f"; "c"**\dir{-};
  ?*!/_2mm/{1};
 "h"; "d"**\dir{-};
  ?*!/_2mm/{1};
 "g"; "f"**\dir{-};
  ?*!/_2mm/{2};
 "i"; "h"**\dir{-};
  ?*!/_2mm/{2};
\endxy
$
\end{center}

\begin{lemma}\label{chain}$n+1$ points of $V_{k+1,n+1}$ form a simplex in the
triangulation of $\Delta_{k+1,n+1}$ if and only if these points form an
$n$-chain in the poset $V$ and the labels of edges are distinct.
Moveover, vertex expressions with the same starting letter will also
have the same ending letter.
\end{lemma}
For example, $HFCBA$ is a simplex in $\Delta_{3,5}$, since the
labels along the path form a permutation $4132$.
\begin{proof}
Starting with a point in $V_{3,5}$, for example $H=(0,1,2,2)$, we need to
add one to each coordinate, in order to get a simplex. And it always
end up with $A=(1,2,3,3)$.
\end{proof}
For each simplex, we define its \emph{vertex expression} to be the
expression formed by its $n+1$ vertices (from small to large in the
poset $V_{k+1,n+1}$). For example, $HFCBA$ is a vertex expression.

We denote the set of all such simplices in their vertex expressions
by $L_{k+1,n+1}$, and denote the corresponding permutations read from the
paths of $V_{k+1,n+1}$ by $R'_{k+1,n+1}$. Since two simplices are adjacent if and only
if their vertices differ by one vertex, we can add a graph structure
on $L_{k+1,n+1}$ (and thus on $R'_{k+1,n+1}$): we connect two simplices if and only if
their vertex expressions differ by one vertex. For example, from
$L_{3,5}$, we get $R'_{3,5}$ by reading the labels of the
corresponding paths in $V_{3,5}$:
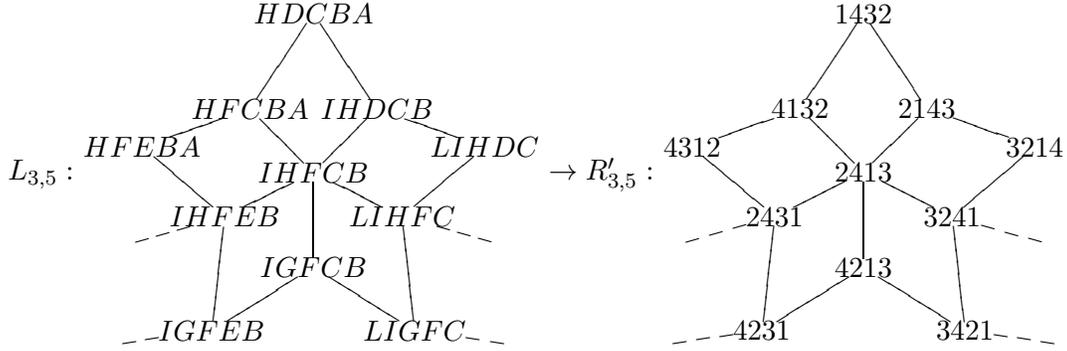
\begin{center}$L_{3,5}:$
$ \xy 0;/r.2pc/: (0,0)*{IHFCB}="a";
 (0,25)*{HDCBA}="b";
 (-10,10)*{HFCBA}="c";
 (10,10)*{IHDCB}="d";
 (-27,4)*{HFEBA}="e";
 (27,4)*{LIHDC}="f";
 (-14,-7)*{IHFEB}="g";
 (14,-7)*{LIHFC}="h";
 (0,-15)*{IGFCB}="i";
 (-16,-25)*{IGFEB}="j";
 (16,-25)*{LIGFC}="k";
   (-23,-16)*{}="gg";
 (23,-16)*{}="hh";
 (-24,-34)*{}="jj";
 (24,-34)*{}="kk";
 "b"; "c"**\dir{-};
 "b"; "d"**\dir{-};
 "c"; "e"**\dir{-};
 "c"; "a"**\dir{-};
 "d"; "a"**\dir{-};
 "d"; "f"**\dir{-};
 "e"; "g"**\dir{-};
 "a"; "g"**\dir{-};
 "a"; "h"**\dir{-};
 "f"; "h"**\dir{-};
 "g"; "j"**\dir{-};
 "a"; "i"**\dir{-};
 "h"; "k"**\dir{-};
 "i"; "j"**\dir{-};
 "i"; "k"**\dir{-};
   (-28,-11)*{}="gg";
 (28,-11)*{}="hh";
 (-30,-27)*{}="jj";
 (30,-27)*{}="kk";
  "jj";"j"**\dir{--};
  "hh";"h"**\dir{--};
 "kk";"k"**\dir{--};
   "gg";"g"**\dir{--};
\endxy
$ $\rightarrow R'_{3,5}:$ $ \xy 0;/r.2pc/: (0,0)*{2413}="a";
 (0,25)*{1432}="b";
 (-10,10)*{4132}="c";
 (10,10)*{2143}="d";
 (-27,4)*{4312}="e";
 (27,4)*{3214}="f";
 (-14,-7)*{2431}="g";
 (14,-7)*{3241}="h";
 (0,-15)*{4213}="i";
 (-16,-25)*{4231}="j";
 (16,-25)*{3421}="k";
  "b"; "c"**\dir{-};
 "b"; "d"**\dir{-};
 "c"; "e"**\dir{-};
 "c"; "a"**\dir{-};
 "d"; "a"**\dir{-};
 "d"; "f"**\dir{-};
 "e"; "g"**\dir{-};
 "a"; "g"**\dir{-};
 "a"; "h"**\dir{-};
 "f"; "h"**\dir{-};
 "g"; "j"**\dir{-};
 "a"; "i"**\dir{-};
 "h"; "k"**\dir{-};
 "i"; "j"**\dir{-};
 "i"; "k"**\dir{-};
   (-28,-11)*{}="gg";
 (28,-11)*{}="hh";
 (-30,-27)*{}="jj";
 (30,-27)*{}="kk";
  "jj";"j"**\dir{--};
  "hh";"h"**\dir{--};
 "kk";"k"**\dir{--};
   "gg";"g"**\dir{--};
\endxy
$
\end{center}
Notice that in $V_{3,5}$, since the vertices $E$, $F$, $H$, $G$,
$I$, $L$ have $z_4=2$, they lie on the lower facet of
$\Delta_{3,5}$. Therefore, we have a dotted line attached to each of
the simplices $IHFEB$, $LIHFC$, $LIGFC$ and $IGFEB$, indicating that
these simplices have a lower facet removed.

We have the following connections between the vertex expressions
(graph $L_{k+1,n+1}$ and $R'_{k+1,n+1}$) and the graphs $R_{k+1,n+1}$ (and $P_{k+1,n+1}$, $Q_{k+1,n+1}$) we studied in
Section 3. For example, compare $R'_{3,5}$ above with $R_{3,5}$ in
Section 3.
\begin{proposition}\label{vertex}$R'_{k+1,n+1}=R_{k+1,n+1}$.
\end{proposition}

\begin{proof}
Since the permutations $r\in R_{k+1,n+1}$ are $\{r\in \mathfrak{S}_{n}\mid
\des(r^{-1})=k\}$, we first need to show that the permutations in
$R'_{k+1,n+1}$ have the same property. For a simplex $\alpha$, let $M_1\dots
M_{n+1}$ be its vertex expression, with $M_1=(m_1,\dots,m_n)$ and
$M_{n+1}=(m'_1,\dots,m'_n)=M_1+\sum_{i=1}^n e_i$. Let
$r'_{\alpha}=a_1a_2\dots a_n$ be the permutation in $R'_{k+1,n+1}$
corresponding to this simplex $\alpha$.  Then we have
$M_{i+1}=M_i+e_{n+1-a_i}$.

Because of the restriction that  $k\le z_n\le k+1$ and $0\le z_1\le
1$ for both $M_1$ and $M_{n+1}$, we have $m_1=0$ and $m_n=k$. By the
other restrictions that $0\le z_{i+1}-z_i\le 1$,  we need to go up
by 1 $k$ times from $m_1$ to $m_n$. So there exists a set
$I\subset[n]$ with $\#I=k$, such that $m_{i+1}=m_i+1$, for each
$i\in I$, and $m_{j+1}=m_{j}$ for  $j\in [n]\backslash I$. To keep
the above restrictions for each $M_i$, $i=1,\dots,n$, we need to add
$e_{i}$ before $e_{i+1}$ for $i\in I$, and add $e_j$  before
$e_{j+1}$ for $j\in [n]\backslash I$. Then by the way we defined
$r'_{\alpha}$, we have $\Des(r'^{-1}_{\alpha})=n+1-I$ and thus
$\des(r'^{-1}_{\alpha})=\#I=k$.

Now we want to show that the edges in the graph $R'_{k+1,n+1}$ are the same as
in $R_{k+1,n+1}$. Since each edge in $L_{k+1,n+1}$ corresponds to a vertex-exchange,
there are two types of edges in $L_{k+1,n+1}$.

First, exchange a vertex in the middle without touching the other
vertices. An edge in $L_{k+1,n+1}$ changing the $i$th vertex with $i\neq 1$
and $i\neq n+1$ corresponds to an edge in $R'_{k+1,n+1}$ exchanging the
$(i-1)$th and the $i$th letters of the permutation $r'\in R'_{k+1,n+1}$. By
the restrictions $0\le z_{j+1}-z_{j}\le 1$, we can make such a
change if and only $r'_{i-1}$ and $r'_i$ are not consecutive
numbers. Therefore, this edge is the  type one edge in $R_{k+1,n+1}$.

Second, remove the first vertex and attach to the end another
vertex. This edge in $L_{k+1,n+1}$ corresponds to the edge in $R'_{k+1,n+1}$ changing
$r'=a_1a_2\cdots a_n$ to $s'=a_2\cdots a_na_1$. We claim that  we
can make such a change if and only if $a_1\neq 1$ and $a_1\neq n$.
In fact, if $a_1=n$, then for the second vertex of the simplex
corresponding to $r'$, we have $z_1=1$. Since the vertex expression
of $s'$ is obtained from that of $r'$ by removing the first vertex
of $r'$ and attaching to the end another vertex, the first vertex of
$s'$ is the same as the second vertex of $r'$. So for the first
vertex of $s'$, we have $z_1=1$, but then we cannot add $e_1$ to
$s'$ any more, since we require $0\le z_1\le 1$; if $a_1=1$, then
$z_n=k+1$ for the first vertex of the simplex corresponding to $s'$,
so we cannot add $e_n$ to $s'$ any more, since we require $k \le
z_n\le k+1$. Therefore, this edge is the type two edge in $R_{k+1,n+1}$.
\end{proof}

\begin{corollary}\label{position}
\begin{enumerate}
\item
Two simplices are in the same big block if and only if the first
vertices in their vertex expression ($L_{k+1,n+1}$) is the same. This implies
that their last vertices are also the same.
\item
Two simplices only differ by the $(i+1)$th vertex in the vertex
expression, if and only if they are connected by an edge $e_i$.
\end{enumerate}
\end{corollary}
For $J\subset [n]$, we call $e_i$ a
 \emph{backward move} if $i\in J$ and $i+1\notin J$; and call it
 a \emph{forward move} if  $i\notin J$ and $i+1\in J$.
Let $t\in s_{I,J}$ for some $I\subset [n]$. When we apply $e_i$ to
$t$, we get a simplex in a
 smaller small block if $e_i$ is a backward move  and in a
 bigger small block if $e_i$ is a forward move. We call both
 backward and forward moves \emph{movable edges}.

For any face $F$ in $\Delta_{k+1,n+1}$, consider the subgraph of $Q_{k+1,n+1}$ with all
simplices containing $F$, denoted by $Q_F$, and its restriction to a
small block $s$, denoted by $Q_{F,s}$.
\begin{lemma}\label{J_0}For any connected small block $s$, $Q_{F,s}$ is
connected. In particular,  $Q_{F,s_{I,J_0}}$ is connected, where
$J_0=\{n-k+1,\dots,n\}$.
\end{lemma}
\begin{proof}
For any two simplices $t_1,t_2\in Q_{F,s}$, let $t_1=M_1\dots
M_{n+1}$
 and $t_2=N_1\dots N_{n+1}$ be their vertex expressions. Since $s$ is connected,
 there exists a path from $t_1$ to $t_2$
without any movable edges. So $M_i=N_{i}$ for all movable edges
$e_{i}$. On the other hand, there exists a path from $t_1$ to $t_2$
using only edges $e_j$ where $M_j\neq N_j$, this path is in $Q_F$.
Since $j$ is not those movable edges, this path is also in $s$, and
thus $t_1$ to $t_2$ is connected by a path in $Q_{F,s}$.

 We only need to show that
 $s_{I,J_0}$ is connected, then by the first statement, $Q_{F,s_{I,J_0}}$ is
connected.

For any fixed big block $I$, each permutation $w\in P_{k+1,n+1}$ is obtained
by a set partition of $[n-k]$ and $J_0$ according to $I$, since
$I=\Des(w)$ and $J_0=\LdDes(w)$. For example, for $n=9$, $k=4$ and
$I=\{1,2,5,6\}$, each $w\in P$ is obtained as follows. We first
choose two from  $J_0=\{6,7,8,9\}$ to be $w_1w_2$ and the other two
to be $w_5w_6$. Within each of the two 2-blocks, numbers need to be
decreasing. Then choose two from  $\{1,2,3,4,5\}$ to be $w_3w_4$ and
the other three to be $w_7w_8w_9$. Within each block, numbers need
to be increasing. Then it is not hard to
 see that any two such permutations can be obtained from each other
without using an $e_{n-k}$-edge, so $s_{I,J_0}$  is connected.
\end{proof}

\begin{proposition}\label{onesmall}$Q_F$ starts with a unique minimal connected small
block.
\end{proposition}
\begin{proof}Suppose not. Let $t_1\in s_{I,J}$, $t_2\in s_{I',J'}$ in two
disconnected minimal small blocks in $Q_F$.
 Write them in vertex expression, we have  $t_1=M_1\dots M_{n+1}$ and $t_2=N_1\dots N_{n+1}$.

 If $I\neq I'$ and they are incomparable, then
there exists another simplex $t\in b_{I''}$ in $Q_F$ with $I''<I'$
and $I''<I$. In fact, looking at the poset $V_{k+1,n+1}$, both $t_1,t_2$ are
some $n+1$-chains in $V_{k+1,n+1}$, their common vertices contain $F$, and
they have different
 ending points $M_{n+1},N_{n+1}$. Let $E\in t_1\cap t_2$ be the maximal
 element of $ t_1\cap t_2$ in $V$, and let $t$ be the chain ending
 at $E$ and passing through $t_1\cap t_2$. Then $t$ has the desired
 property. So $t_1,t_2$ are not in  minimal small blocks.

Now we assume $I=I'$. If $J=J'$, then by Lemma \ref{J_0}, $J\neq
J_0$, so $J$ has a backward move. We can show that there exists a
backward move $i$ of $J$ such that $M_i\neq N_i$. First, it is easy
to see that there exists a movable edge $e_i$ such that $M_i\neq
N_i$, otherwise $s_{I,J}$ is connected. Then by symmetry, it is
impossible that all of these movable edges are forward moves. Then
let $t$ be the simplex obtained from $t_1$ by an $e_i$ move. Since
$M_i\neq N_i$, we have $M_i\notin F$. Therefore, $t\in Q_F$ and $t$
is in a smaller small block, which  contradicts the assumption that
$s_{I,J}$ is a minimal small block in $Q_F$.

Now assume $J\neq J'$ and they are incomparable. By Lemma
\ref{desld} (part two), we need to apply a sequence of moves to get
from  $s_{I,J}$ to $s_{I,J'}$. Since $J, J'$ are incomparable, there
exists a backward move for $J$, which is a necessary move from
$s_{I,J}$ to $s_{I,J'}$. It follows that there exists such a move
$e_i$ with $M_i\neq N_i$.  Then we can apply this move to $t_1$ and
get a smaller small block in $Q_F$ than $s_{I,J}$.

\end{proof}

\section{Proof of Theorem~\ref{cover}: second shelling}
We want to show that the $h^*$-polynomial of $\Delta'_{k+1,n+1}$ is also given by
$$\sum_{\substack{w\in\, \mathfrak{S}_{n}\\
\des(w)=k}}t^{\cover(w)},$$
we will define $\cover$ in a minute.
Compare this to Theorem~\ref{Shelling and Ehrhart polynomial}: if $\Delta'_{k+1,n+1}$ has a shellable unimodular triangulation $\Gamma_{k+1,n+1}$, then its $h^*$-polynomial is
$$\sum_{\alpha\in\,
\Gamma_{k+1,n+1}}t^{\#(\alpha)}.$$
Similar to Theorem~\ref{exc}, we will define shellable unimodular triangulation for $\Delta'_{k+1,n+1}$, but this shelling is different from the one we use for Theorem~\ref{exc}. Label each simplex $\alpha\in \Gamma_{k+1,n+1}$ by a permutation $w_{\alpha}\in \mathfrak{S}_n$ with $\des(w_{\alpha})=k$. Then show that $\#(\alpha)=\cover(w_{\alpha})$.

We start from the graph $\Gamma_{k+1,n+1}$ studied in Section 3.3. Define a
graph $M_{k+1,n+1}$  such that $v\in V(M_{k+1,n+1})$ if and only if $v^{-1}\in
V(\Gamma_{k+1,n+1})$ and $(u,v)\in E(M_{k+1,n+1})$ if and only if $(u^{-1},v^{-1})\in
E(\Gamma_{k+1,n+1})$. By Proposition \ref{gamma}, we have
$$V(M_{k+1,n+1})=\{w\in \mathfrak{S}_{n}\mid \des(w)=k\},$$ and
$(w,u)\in E(M_{k+1,n+1})$ if and only if $w$ and $u$ are related in one of the
following ways:
\begin{enumerate}
\item type one: exchanging the letters $i$ and $i+1$ if these two letters are not adjacent in $w$ and $u$
\item type two: one is obtained by subtracting 1 from each letter of the other
(1 becomes $n-1$).
\end{enumerate}
 Now we want to orient the edges of $M_{k+1,n+1}$ to make it a digraph. Consider $e=(w,u)\in E(M_{k+1,n+1})$.
\begin{enumerate}
\item if  $e$ is of type one, and $i$ is before $i+1$ in $w$, i.e.,
$\inv(w)=\inv(u)-1$, then orient the edge as $w\leftarrow u$.
\item if edge $(w,u)$ is of type two, and $v$ is obtained by subtracting 1 from each letter
of $u$ (1 becomes $n-1$), then orient the edge as $w\leftarrow u$.
\end{enumerate}
\begin{example}\label{order} Here is the directed graph $M_{3,5}$ for
$\Delta'_{3,5}$:
$$
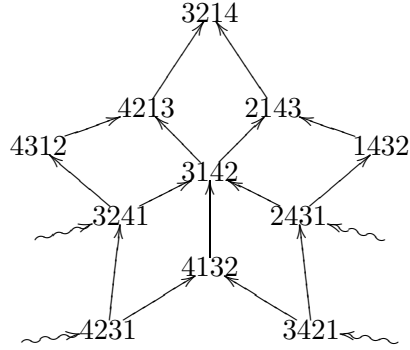
\xy 0;/r.2pc/: (0,0)*{3142}="a";
 (0,25)*{3214}="b";
 (-10,10)*{4213}="c";
 (10,10)*{2143}="d";
 (-27,4)*{4312}="e";
 (27,4)*{1432}="f";
 (-14,-7)*{3241}="g";
 (14,-7)*{2431}="h";
 (0,-15)*{4132}="i";
 (-16,-25)*{4231}="j";
 (16,-25)*{3421}="k";
 (-28,-11)*{}="gg";
 (28,-11)*{}="hh";
 (-30,-27)*{}="jj";
 (30,-27)*{}="kk";
 {\ar "c";"b"};%
 {\ar "d";"b"};%
 {\ar "e";"c"};%
 {\ar "a";"c"};%
 {\ar "a";"d"};%
 {\ar "f";"d"};%
 {\ar "g";"e"};%
 {\ar "g";"a"};%
 {\ar "h";"a"};%
 {\ar "h";"f"};%
 {\ar "i";"a"};%
 {\ar "j";"g"};%
 {\ar "j";"i"};%
 {\ar "k";"i"};%
 {\ar "k";"h"};%
 {\ar@{~>} "jj";"j"};%
  {\ar@{~>} "hh";"h"};%
   {\ar@{~>} "kk";"k"};%
    {\ar@{~>} "gg";"g"};%
\endxy
$$
\end{example}
\begin{lemma}\label{acyc}There is no cycle in the directed graph
$M_{k+1,n+1}$.
\end{lemma}
\begin{proof} Let us call the subgraph of $M_{k+1,n+1}$ connected by only type one edges
 a component. Then there is no cycle involving type two edges
since they all point in the same direction from one component to
another. Then there is no cycle involving only type one edges
either, since the number of inversions decreases along the directed
path of type one edges.
\end{proof}

Therefore, $M_{k+1,n+1}$ defines a poset on $V(M_{k+1,n+1})$ and $M_{k+1,n+1}$ is the Hasse
diagraph of the poset, which we
 still denote as $M_{k+1,n+1}$. This poset can be seen as a variation
of the poset of the weak Bruhat order.

For an element in the poset $M_{k+1,n+1}$, the larger its rank is, the further
its corresponding simplex is from the origin. More precisely, notice
that each $v=(x_1,\dots,x_n)\in V_{k+1,n+1}=\Delta_{k+1,n+1}\cap\Z^n$ has
$|v|=\sum_{i=1}^n x_i=k$ or $k+1$. For $u\in M_{k+1,n+1}$, by which we mean
$u\in V(M_{k+1,n+1})$, define
$$A_u=\#\{v \text { is a vertex of the simplex } s_{u^{-1}}\mid
|v|=k+1\}.$$
\begin{proposition}Let $w>u$ in the above poset  $M_{k+1,n+1}$. Then $A_w\ge A_u$.
\end{proposition}
This proposition follows from the following lemma and the definition
of the two types of directed edges.
\begin{lemma}\label{increase}$A_u=u_{n}$.
\end{lemma}
\begin{proof}
Let $w=u^{-1}$ and use the notations in section 2. Vertices of
$s_{w}$ are  $\phi(v_i)$ for $i=0,\dots,n$. Since $v_0=(0,\dots,0)$,
by (\ref{stanmap}), $|\phi(v_0)|=k$, so $x_{n+1}=1$ for $\phi(v_0)$.
By Lemma \ref{move}, from $\phi(v_{n-u_n})$ to $\phi(v_{n-u_n+1})$,
$x_nx_{n+1}$ is changed from $01$ to $10$. Moreover $x_{n+1}=1$,
thus $|\phi(v_i)|=\sum_{j=1}^n x_j=k$ for $i=0,\dots,n-u_n$, and
$x_{n+1}=0$, thus $|\phi(v_i)|=k+1$  for $i=n-u_n+1,\dots,n$.
Therefore, there are $u_n$ vertices with $|\phi(v_i)|=k+1$, thus
$A_u=u_n$.
\end{proof}

We define \emph{cover} of a permutation $w\in M_{k+1,n+1}$ to be the number
of permutations $v\in M_{k+1,n+1}$ it covers, i.e., the number of incoming
edges of $w$ in the graph $M_{k+1,n+1}$. From the above definition, we have
the following, (in the half-open setting):
\begin{lemma}\label{defcover}
\begin{enumerate}
\item If $w_1=1$, then $\cover(w)=\#\{i\in[n-1]\mid
(w^{-1})_i+1<(w^{-1})_{i+1}\}$;
\item if $w_1\neq 1$, then
$\cover(w)=\#\{i\in[n-1]\mid (w^{-1})_i+1<(w^{-1})_{i+1} \}+1$.
\end{enumerate}
\end{lemma}
\begin{proof} The elements in $\{i\in[n-1]\mid
(w^{-1})_i+1<(w^{-1})_{i+1}\}$ correspond to the type one edges
pointing to $w$. So we need to show that $w$ has an incoming type
two edge in the graph for $\Delta'_{k,n}$ if and only if $w_1\neq
1$. Let $u$ be the permutation obtained by subtracting one from each
letter of $w$ (1 becomes $n-1$).
\begin{enumerate}
\item If  $w_1\neq 1$ and $w_{n-1}\neq 1$, then
$\des(u)=\des(w)$, so $u\in M_{k,n}$.
\item If $w_{n-1}=1$, then
$\des(u)=\des(w)-1$, so $u\in M_{k-1,n}$. Since we are considering
the half-open setting, this incoming edge is still in
$\Delta'_{k,n}$. This corresponds to the waved edges in the above
example of $\Delta'_{3,5}$.
\item If $w_1=1$, then $\des(u)=\des(w)+1$, so this edge is not
in $\Delta'_{k,n}$.\qedhere
\end{enumerate}
\end{proof}

Recall the graph $R_{k+1,n+1}$ defined in Section 4 is obtained by
$$M_{k+1,n+1}\xrightarrow{w^{-1}}\Gamma_{k+1,n+1}\xrightarrow{\rev}R_{k+1,n+1}.$$
By Proposition~\ref{vertex}, $R_{k+1,n+1}$ is also obtained from the $n-$chain
expression of each simplex in $\Delta_{k+1,n+1}$. We can describe the same orientation of edges $(u,w)$ in $R_{k+1,n+1}$ with
$n-$chain expression $u=L_1<\dots<L_{n+1}$ and
$w=I_1<\dots<I_{n+1}$:
\begin{enumerate}
\item type one edge $e_i$: if $u_{i}<u_{i+1}$, then $u\leftarrow w$.
We have $L_{i+1}\neq I_{i+1}$  with $\rank(L_{i+1})=\rank(I_{i+1})$
in the poset $V$ and $L_j=I_j$ for all $j\neq i+1$. $u_{i}<u_{i+1}$
if and only if the vector
$L_{i+1}=(z_1,\dots,z_n)<I_{i+1}=(z'_1,\dots,z'_n)$ in dominance
order, i.e., $z_n+\dots+z_{n-\ell}\ge z'_n+\dots+z'_{n-\ell}$ for
all $\ell$. Note that by definition, we have $z_n\ge
z_{n-1}\ge\dots\ge z_1$ and $z'_n\ge z'_{n-1}\ge\dots\ge z'_1$.
\item type two edge:  if  $w=u_2\dots u_{n}u_{1}$, then $w\leftarrow u$.
  This corresponds to the case $w=L_2<\dots<L_{n+1}<L_1$ in the poset $V_{k+1,n+1}$.
\end{enumerate}

With the above ordering on the $n$-chain expressions of simplices in
$\Delta_{k+1,n+1}$, we can prove the following:
\begin{proposition}\label{shelling2}Any
linear extension of the above ordering gives a shelling order on the
triangulation of $\Delta'_{k+1,n+1}$.
\end{proposition}

\begin{proof}We want to show that for any linear extension of the
order in $M_{k+1,n+1}$, every simplex has a unique minimal nonface (see
definitions in Section 2.3).

For each simplex $\alpha\in \Delta_{k+1,n+1}$, assign to it a face $F\subset
\alpha$ in the following way. Each incoming edge
$\alpha\xleftarrow{e_i} \alpha_i$ defines a unique vertex $L_i$ of
$\alpha$ that $\alpha $ has but $\alpha_i$ does not have. Then let
$F=\{L_i\}$ be given by all the incoming edges of $\alpha$. We want
to show that $F$ is the unique minimal face of $\alpha$ and it has
never appeared before in any linear extension of the ordering given
by the directed graph.

First, assume $F$ has never appeared before, then it is clear that
$F$ is the unique minimal face, i.e., any proper subface of $F$ has
appeared before. In fact, let $L_i$ be a vertex in $F$ but not in
$F'\subset F$. Then we have $F'\subset \alpha_i$ since $\alpha_i$
has every vertex of $\alpha$ except for $L_i$.

Now we will show that $F$ has never appeared before $\alpha$ in any
linear extension, i.e., for any other $\beta$ which also has $F$,
there exists a directed path from $\alpha$ to $\beta$. It suffices
to show the following: for any face $F\subset \Delta_{k+1,n+1}$, the component
$M_F$ of simplices containing $F$ has a unique source, and any other
simplex is reachable from that source (there exists a directed path
from $\alpha$ to $\beta$).

In $M_F$, let us first consider the subgraph of simplices starting
with the same letter, say $A$, denoted by $M_{F,A}$. We want to
prove that $M_{F,A}$ has a unique source, and any other simplex is
reachable from that source. By the description of edges in $M_{k+1,n+1}$, simplices in
$M_{F,A}$ are connected by type one edges. For any edge $e_i$ $
H=H_1\dots H_{n+1}\rightarrow W=W_1\dots W_{n+1}$, we have
 $i\neq 0,n$, $H_{i+1}\neq W_{i+1}$ and $H_j=W_j$ for all $j\neq
 i+1$. Now let $F\cup \{A,B\}=\{F_1<F_2\dots <F_{\ell}\}$ ordered as in the poset $V_{k+1,n+1}$.
 It is clear that all simplices $M_{F,A}$ are $(n+1)$-chain in the
 interval $[A,B]$, where $B=A+\sum_{i=1}^n e_i$ passing through $F_1,\dots,F_{\ell}$.
 Now order the letters of the same rank in each of the intervals $[F_i,F_{i+1}]$
 by dominance order. We claim that the unique source is the chain obtained by
 choosing the dominant maximal element in each rank. First, notice
 that in the interval $[F_i,F_{i+1}]$, if
 $\rank(A_1)=\rank(A_2)+1=k$ and both $A_1$ and $A_2$ are maximal in
 dominance order compared to other element in $[F_i,F_{i+1}]$ with
 ranks $k$ and $k-1$ respectively, then we have $A_1>A_2$. So the dominant maximal elements in each rank of $[F_i,F_{i+1}]$
 and $F\cup\{A,B\}$ form a chain. Moreover, for any other
 chain in $M_{F,A}$, we can apply a simple move to change one vertex to a larger
 element in dominant order until we reach the chain with dominant
 maximal in each rank. Then the reachability also follows.

Now consider the whole $M_F$. We claim that the ending point of the
source is the maximal element in $F$, denoted by $F_h$. Any chain
$\beta$ not ending with $F_{h}$ ends with some letter larger than
$F_{h}$ in the poset $V_{k+1,n+1}$, then by moving down steps, there exists a
simplex
 $\gamma\in M_{F,F'_{h}}$, where $F'_{h}=F_{h}-\sum_{i=1}^n e_i$ such that there is a directed path from
 $\gamma$ to $\beta$. We know that $M_{F,F'_{h}}$ has its unique
 source $\alpha$, which connects to $\gamma$ by a directed path
 towards $\gamma$. Thus we have a directed path from $\alpha$ to
 $\beta$ via $\gamma$.
\end{proof}

 It is clear that the shelling number of the simplex corresponding to $w$ is
$\cover(w)$. Then by Theorem \ref{Shelling and Ehrhart polynomial}
and Proposition \ref{shelling2}, we have a proof of Theorem~\ref{cover}. Combine the above with
Theorem \ref{exc}, we have an indirect proof
of Corollary~\ref{equal}.

We want a direct combinatorial proof, which will give another proof
of Theorem \ref{exc}, and help us find a colored version of
excedance by Theorem \ref{color} in the next section.


\section{The $h^*$-polynomial for generalized half-open hypersimplex}
We want to extend Theorem \ref{cover} to the hyperbox
$B=[0,a_1]\times \dots \times [0,a_n]$. Write $\alpha=(a_1,\dots,a_d)$
and define the generalized half-open hypersimplex as
\begin{equation}\label{general}
\Delta'_{k,\alpha}=\{(x_1,\dots,x_n)\mid 0\le x_i\le a_i;  k-1<x_1+\dots+x_n\le k\}.
\end{equation}
Note that the above polytope is a multi-hypersimplex studied in \cite{lp}. For a nonnegative integral vector $\beta=(b_1,\dots,b_n)$, let
$C_{\beta}=\beta+[0,1]^n$ be the cube translated from the unit cube
by the vector $\beta$. We call $\beta$  the \emph{color} of
$C_{\beta}$.

We extend the triangulation of the unit cube to $B$ by translation
and  assign to each simplex in $B$ a \emph{colored permutation}
$$w_{\beta}\in \mathfrak{S}_{\alpha}=\{w\in \mathfrak{S}_n\mid
b_i<a_i, i=1,\dots,n\}.$$

Let $F_i=\{x_i=0\}\cap[0,1]^n$ for $i=1,\dots,n$. Define the
 \emph{exposed facets} for the simplex $s_{u^{-1}}$ in $[0,1]^n$,
 with $u\in M$, to be
 $\Expose(u)=\{i\mid s_{u^{-1}}\cap F_i \text{ is a facet of }s_{u^{-1}}\}$.

 We can compute $\Expose(u)$ explicitly
 as follows
\begin{lemma}Set $u_0=0$. Then $\Expose(u)=\{i\in[n]\mid
u_{i-1}+1=u_i\}$.
\end{lemma}
\begin{proof} Denote $u^{-1}=w$. Let
$\phi(v_i)$, $i=0,\dots,n$ be the vertices of $s_w$. Then $i\in
\Expose(u)$ if and only if $x_i=0$ for $n$ vertices of $s_w$. By the
description of vertices of $s_w$ in Lemma \ref{move}, from
$\phi(v_{n-u_i})$ to $\phi(v_{n-u_i+1})$, we change $x_{i}x_{i+1}$
from $01$ to $10$; and from $\phi(v_{n-u_{i-1}})$ to
$\phi(v_{n-u_{i-1}+1})$, we change $x_{i-1}x_{i}$ from $01$ to $10$.
If $u_{i-1}+1=u_i$, we have $v_{n-u_{i-1}}=v_{n-u_i+1}$. Then 1 will
pass through $x_i$ quickly and thus $x_i=1$ for only one vertex
$\phi(v_{n-u_i+1})$ of $s_w$. Otherwise, $x_i=1$ for more than one
vertex.
\end{proof}
Now we want to extend the shelling on the unit cube to
 the larger rectangle. In this extension, $F_i$ will be removed from $C_\beta$ if $b_i\neq
0$. Therefore, for the simplex $s_{w_{\beta}}$, we will remove the
facet $F_i\cap s_{w_{\beta}}$ for each $i\in \Expose(w)\cap\{i\mid
b_i\neq 0\}$ as well as the $\cover(w_{\beta})$ facets for neighbors
within $C_{\beta}$. We call this set $\Expose(w)\cap\{i\mid b_i\neq
0\}$ the  \emph{colored exposed facet (cef)}, denoted by
$\cef(w_{\beta})$, for each colored permutation
$w_{\beta}=(w,\beta)$.

Based on the above extended shelling, with some modifications of
Proposition \ref{shelling2}, we can show that the above order is a
shelling order. We show the idea of the proof by the following
example.

\begin{example}Consider $\Delta'_{k,\alpha}$ for $\alpha=(1,2,2)$ and $k=3$. In $z$-coordinates, where $z_i=x_1+ \dots+x_i$,
we have
$$V_{3,\{1,2,2\}}=\{A(0,0,2),\,B(0,1,2),\,C(1,1,2),\,F(0,2,2),\,G(1,2,2),$$
$$D(0,1,3),\,E(1,1,3),\,H(0,2,3),\,I(1,2,3),\,L(1,3,3)\}.$$
Drawing them in the poset as described in Section 5, we have the
following poset on the left. The simplices in the triangulation of
$\Delta_{k,\alpha}$ are 3-chains of $V_{3,\{1,2,2\}}$ with distinct labels along
the chain. We draw these 3-chains on the right with an edge between
each pair of adjacent simplices.

If two simplices are in the same cube, then we orient the edges as
in Section 3. If not, then the arrow points to the one whose
permutation has fewer descents. With this extension, we can still
compare two simplices that only differ by the $(i+1)$th vertices
$L_{i+1}$ and $I_{i+1}$ by comparing $L_{i+1}$ and $I_{i+1}$ in the
dominance order. So the proof of Proposition \ref{shelling2} holds
for $\Delta_{k,\alpha}$ too.
\begin{center}
 $ \xy 0;/r.37pc/:
 (0,0)*{A}="a";
 (0,10)*{B}="b";
 (-13,15)*{C}="c";
  (10,15)*{D}="d";
 (0,25)*{F}="f";
 (-3,20)*{E}="e";
  (-13,30)*{G}="g";
   (10,30)*{H}="h";
    (-3,35)*{I}="i";
     (-3,45)*{L}="l";
 "a"; "b"**\dir{-};
  ?*!/_2mm/{2};
 "b"; "c"**\dir{-};
  ?*!/_2mm/{3};
 "b"; "d"**\dir{-};
  ?*!/_-2mm/{1};
 "c"; "e"**\dir{-};
  ?*!/_2mm/{1};
 "e"; "d"**\dir{-};
  ?*!/_2mm/{3};
 "g"; "c"**\dir{-};
  ?*!/_2mm/{2};
 "i"; "e"**\dir{-};
  ?*!/_2mm/{2};
 "h"; "d"**\dir{-};
  ?*!/_2mm/{2};
 "f"; "b"**\dir{--};
  ?*!/_-2mm/{2};
 "l"; "i"**\dir{-};
  ?*!/_2mm/{2};
 "g"; "i"**\dir{-};
  ?*!/_2mm/{1};
 "i"; "h"**\dir{-};
  ?*!/_2mm/{3};
  "f"; "g"**\dir{--};
  ?*!/_2mm/{3};
  "f"; "h"**\dir{--};
  ?*!/_-2mm/{1};
\endxy
\,\,\,$ $ \xy 0;/r.3pc/: (0,0)*{BCGI}="a";(0,-3)*{321};
 (-20,0)*{ABCE}="b";(-20,-3)*{231};
 (-30,10)*{ABDE}="c";(-33,7)*{213};
 (-10,10)*{BCEI}="d";(-10,7)*{312};
 (10,10)*{BFGI}="e";(10,7)*{231};
 (-20,20)*{BDEI}="f";(-20,17)*{132};
 (0,20)*{BFHI}="g";(0,17)*{213};
 (20,20)*{FGIL}="h";(20,17)*{312};
 (-10,30)*{BDHI}="i";(-10,27)*{123};
 (10,30)*{FHIL}="l";(10,27)*{132};
   (0,-10)*{}="aa";
 (-20,-10)*{}="bb";
 (10,0)*{}="ee";
 {\ar "a";"d"};%
 {\ar "a";"e"};%
  {\ar "b";"d"};%
   {\ar "b";"c"};%
    {\ar "c";"f"};%
     {\ar "d";"f"};%
      {\ar "e";"g"};%
       {\ar "e";"h"};%
        {\ar "f";"i"};%
         {\ar "g";"i"};%
          {\ar "g";"l"};%
          {\ar "h";"l"};%
 {\ar@{~>} "aa";"a"};%
  {\ar@{~>} "bb";"b"};%
   {\ar@{~>} "ee";"e"};%
\endxy
$
\end{center}
\end{example}

Then, by Theorem \ref{Shelling and Ehrhart polynomial} and the fact
that the shelling number for $w_{\beta}$ is
$\cover(w_{\beta})+\cef(w_{\beta})$, we have the following theorem.
\begin{theorem}\label{color}
The $h^*$-polynomial for $\Delta'_{k,\alpha}$ is $$\sum_{\substack{w_{\beta}\in \mathfrak{S}_{\alpha}\\
\des(w)+|\beta|=k-1}}t^{\cover(w_{\beta})+\cef(w_{\beta})}.$$

\end{theorem}
\begin{example}\label{tables}Consider $n=5$, $k=5$ and $\alpha=(1,2,2,4)$. We want
to compute the $h^*$-polynomial of $\Delta'_{5,(1,2,2,4)}$ by Theorem
\ref{color}, where the sum is over all $(w,\beta)$ with $w\in
\mathfrak{S}_4$, $\beta=(b_1,\dots,b_4)$ with $b_1=0$, $0\le b_2<2$,
$0\le b_3<2$, $0\le b_4<4$ and $\des(w)+|\beta|=4$.
\begin{enumerate}
\item If $\des(w)=0$, we have $w=1234$, and the color $\beta$ with
 $|\beta|=4$ is one of  $(0,0,1,3)$, $(0,1,0,3)$ and $(0,1,1,2)$.

 \begin{table}[ht]\label{0}
 \caption{$\des(w)=0$}
\centering 
\begin{tabular}{c|c|c|cc c} 
\hline 
$w$ & $\cover(w)$ & $\Expose(w)$ & $\cef(w_{(0,0,1,3)})$ & $\cef(w_{(0,1,0,3)})$ & $\cef(w_{(0,1,1,2)})$\\ [0.5ex] 
\hline\hline 
1234 & 0 & $\{1,2,3,4\}$ & 2 &2&3\\ \hline
\end{tabular}
\end{table}
From  Table \ref{0}, we have $\sum_{
\des(w)=0,\,|\beta|=4}t^{\cover(w_{\beta})+\cef(w_{\beta})}=2t^2+t^3.$

\item If $\des(w)=1$, the color $\beta$ with
 $|\beta|=3$ is one of  $(0,0,0,3)$, $(0,0,1,2)$, $(0,1,0,2)$ and $(0,1,1,1)$.
\begin{table}[ht]\label{1}
 \caption{$\des(w)=1$}
\centering 
\begin{tabular}{c|c|c|cc cc} 
\hline 
$w$ & $\cover(w)$ & $\Expose(w)$ & $\cef(w_{(0,0,0,3)})$ & $\cef(w_{(0,0,1,2)})$ & $\cef(w_{(0,1,0,2)})$ & $\cef(w_{(0,1,1,1)})$\\ [0.5ex] 
\hline 
\hline 1243 & 1 &$\{1,2\}$& 0 &0&1&1\\
\hline 1342 & 1 &$\{1,3\}$& 0 &1&0&1\\
\hline 1423 & 1 &$\{1,4\}$& 1 &1&1&1\\
\hline 2341 & 1 &$\{2,3\}$& 0 &1&1&2\\
\hline 3412 & 1 &$\{2,4\}$& 1 &1&2&2\\
\hline 4123 & 1 &$\{3,4\}$& 1 &2&1&2\\
\hline 1324 & 2 &$\{1\}$  & 0 &0&0&0\\
\hline 2314 & 2 &$\{2\}$  & 0 &0&1&1\\
\hline 3124 & 2 &$\{3\}$  & 0 &1&0&1\\
\hline 2134 & 2 &$\{4\}$  & 1 &1&1&1\\
\hline 2413 & 2 &$\{\}$   & 0 &0&0&0\\
\hline
\end{tabular}
\end{table}
From Table \ref{1}, we have $\sum_{
\des(w)=1,\,|\beta|=3}t^{\cover(w_{\beta})+\cef(w_{\beta})}=5t+26t^2+13t^3$.

\item If $\des(w)=2$,  the color $\beta$ with
 $|\beta|=2$ is one of  $(0,0,0,2)$, $(0,0,1,1)$, $(0,1,0,1)$ and $(0,1,1,0)$.
\begin{table}[ht]\label{2}
 \caption{$\des(w)=2$}
\centering 
\begin{tabular}{c|c|c|cc cc} 
\hline 
$w$ & $\cover(w)$ & $\Expose(w)$ & $\cef(w_{(0,0,0,2)})$ & $\cef(w_{(0,0,1,1)})$ & $\cef(w_{(0,1,0,1)})$ & $\cef(w_{(0,1,1,0)})$\\ [0.5ex] 
\hline 
\hline 1432 & 1 &$\{1\}$ & 0 &0&0&0\\
\hline 3421 & 1 &$\{2\}$ & 0 &0&1&1\\
\hline 4231 & 1 &$\{3\}$ & 0 &1&0&1\\
\hline 4312 & 1 &$\{4\}$ & 1 &1&1&0\\
\hline 2143 & 2 &$\{\}$  & 0 &0&0&0\\
\hline 2431 & 2 &$\{\}$  & 0 &0&0&0\\
\hline 3214 & 2 &$\{\}$  & 0 &0&0&0\\
\hline 3241 & 2 &$\{\}$  & 0 &0&0&0\\
\hline 4132 & 2 &$\{\}$  & 0 &0&0&0\\
\hline 4213 & 2 &$\{\}$  & 0 &0&0&0\\
\hline 3142 & 3 &$\{\}$  & 0 &0&0&0\\
\hline
\end{tabular}
\end{table}
From Table \ref{2}, we have $\sum_{
\des(w)=2,\,|\beta|=2}t^{\cover(w_{\beta})+\cef(w_{\beta})}=9t+31t^2+4t^3$.

\item If $\des(w)=3$, we have $w=4321$, and the color $\beta$ with
 $|\beta|=1$ is one of  $(0,1,0,0)$, $(0,0,1,0)$ and $(0,0,0,1)$.
 \begin{table}[ht]\label{3}
 \caption{$\des(w)=3$}
\centering 
\begin{tabular}{c|c|c|cc c} 
\hline 
$w$ & $\cover(w)$ & $\Expose(w)$ & $\cef(w_{(0,1,0,0)})$ & $\cef(w_{(0,0,1,0)})$ & $\cef(w_{(0,0,0,1)})$\\ [0.5ex] 
\hline\hline 
4321 & 1 & $\{\}$ & 0 &0&0\\ \hline
\end{tabular}
\end{table}
From  Table \ref{3}, we have $\sum_{
\des(w)=3,\,|\beta|=1}t^{\cover(w_{\beta})+\cef(w_{\beta})}=3t.$
\end{enumerate}
To sum up, the $h^*$-polynomial of $\Delta'_{5,(1,2,2,4)}$ is
$17t+59t^2+18t^3$.
\end{example}
\section{Some identities}
\begin{proposition}For any $k\in[n-1]$, we have
\begin{enumerate}
\item $\#\{w\in\mathfrak{S}_{n}\mid\exc(w)=k,\des(w)=1\}=\binom{n}{k+1}$.
\item $\{w\in\mathfrak{S}_{n}\mid
\des(w)=k,\cover(w)=1\}=\{w\in\mathfrak{S}_{n}\mid
\#\Expose(w)=n-(k+1)\}.$
\item $\#\{w\in\mathfrak{S}_{n}\mid
\des(w)=k,\cover(w)=1,\Expose(w)=S\}=1$, for any $S\subset[n]$ with
$|S|=n-(k+1)$.
\item $\#\{w\in\mathfrak{S}_{n}\mid
\des(w)=k,\cover(w)=1\}=\binom{n}{k+1}$.
\end{enumerate}
\end{proposition}

\begin{proof}\begin{enumerate}
\item Notice that if $i$ is an exceedance and $i+1$ is not, then $i$
is a descent. Since $\des(w)=1$, all exceedances are next to each
other. Let $i$ be the first exceedance. Then it suffices to choose
$i<w_i<w_{i+1}<\dots<w_{n-k+1}$ to determine $w$.
\item Let $i_0$ be the smallest $i$ such that $i\notin
\Expose(w)$. Notice that this $i_0$ will cause one $\cover$. In
fact, if $i_0=1$, then $w_1\neq 1$; if $i_0>1$, then $w_{i_0}-1$ is
before $w_{i_0}$ and they are not adjacent. Since $\cover(w)=1$,
after the $i_0$th position of $w$, there is no $j\dots(j+1)$. Then
it follows that for each $i\notin \Expose(w)$ with $i\neq i_0$,
$i-1$ is a descent of $w$. On the other hand, if $j\in \Expose(w)$,
$j-1$ is not a descent. Therefore, to make $\des(w)=k$, we need $k$
elements other than $i_0$ that are not in $\Expose(w)$.
\item Let $S=\{a_1,\dots,a_{k+1}\}$. It is easy to check that the
only $w$ satisfying the condition is the following: $w_1\dots
w_{a_1-1}=1 \dots (a_1-1)$,  $w_{a_1}>w_{a_2}>\dots>w_{a_{k+1}}$ and
$w_{j+1}=w_j+1$ for $j=a_i,a_i+1,\dots, a_{i+1}-2$ if
$a_{i+1}-a_i>1$ for $i=1,2,\dots,k+1$, where we set $a_{k+1}=n+1$.
For example, if $S=\{2,3,5,7\}$ for $n=9$, then $w=197856234$.
\item Follows from (2) and (3).\qedhere
\end{enumerate}
\end{proof}
\begin{proposition}For any $1<i<n$, we have
\begin{enumerate}
\item $\#\{w\in\mathfrak{S}_{n}\mid\exc(w)=1,\des(w)=k\}=\binom{n+1}{2k}$.
\item $\#\{w\in\mathfrak{S}_{n}\mid
\des(w)=1,\#\Expose(w)=n-2k \text{ or }n+1-2k\}=1$
\item $\{w\in\mathfrak{S}_{n}\mid
\des(w)=1,\#\Expose(w)=n-2k \text{ or
}n+1-2k\}\subset\{w\in\mathfrak{S}_{n}\mid \cover(w)=k\}$.
\item $\#\{w\in\mathfrak{S}_{n}\mid
\des(w)=1,\cover(w)=k\}=\binom{n}{2k}+\binom{n}{2k-1}=\binom{n+1}{2k}$.
\end{enumerate}
\end{proposition}
\begin{proof}
\begin{enumerate}
\item Let the unique exceedance be $i$ and assume $w_i=j>i$. First,
we have $w_{\ell}=\ell$ for $\ell<i$ and $\ell>j$, also $w_{\ell}\le
\ell$ for $i<\ell<j$. Now notice that if $i<\ell\in \Des(w)$, then
we must have $w_{\ell}=\ell$, otherwise, we cannot have $w_h\le h$
for all $i<h<\ell$. Then, we can show that a $2k$-subset
$\{i<i_1<j_1+1<i_2<j_2+1<\dots<i_{k-1}<j_{k-1}+1<j+1\}\subset[n+1]$
corresponds to a unique such permutation $w$ in the following way:
$w_{s}=s$ for $i_{\ell}\le s\le j_{\ell}$, for all $1\le \ell \le
k-1$ and then fill the gaps with the left numbers increasingly. We
see that $\Des(w)=\{i,j_1,j_2,\dots,j_{k-1}\}$. For example,
consider $\{2,3,4,6,8,9\}$ for $n=9$. First we have $w_1=1$,
$w_9=9$; then we have $w_2=8$, $w_3=3$, $w_6w_7=67$. Finally we fill
the positions $w_4,w_5,w_8$ with the rest of the numbers $2,4,5$,
and get $w=183246759$ with $\exc(w)=1$ and $\Des(w)=\{2,3,7\}$.
Conversely, it is easy to define a unique $2k$-subset as above for a
given $w$.
\item Let $[n]-\Expose(w)=\{i_1,\dots,i_{\ell}\}$, where $\ell=2k-1$
or $2k$. It is not very hard to see that in order to make sure
$\des(w)=1$, $w$ has to be the following one. Define $w_i=i$ for
$1\le i<i_1$. Then let $r=\lfloor\frac{\ell}{2}\rfloor$, define
$B_j=w_{i_j}\dots w_{i_{j+1}-1}$ for $1\le j\le r$ and
$A_j=w_{i_{r+j}}\dots w_{i_{r+j+1}-1}$ for $1\le j\le \ell-r$, where
we set $i_{\ell+1}=n+1$. Then we put numbers $i_1,i_1+1,\dots, n$
into the positions $A_1B_1A_2B_2\dots A_rB_r (A_{r+1})$
alternatively. For example,  Let $[n]-\Expose(w)=\{3,4,6,8,9\}$ with
$n=9$. Then $w=125893467$.
\item It is clear from the construction in (2), that $w$ has $k$
covers.
\item Follows from (2) and (3).\qedhere
\end{enumerate}
\end{proof}

See the relations between cover and  Exposed set shown in Tables
\ref{1} and \ref{2} for an example of the above two propositions.
\subsection*{Acknowledgements.}
 I thank my advisor Richard Stanley for introducing me to the problem
 and giving me help and encouragement, Ira Gessel for helpful communication about
the generating function proof,  Yan Zhang for a nice idea for the
proof of Lemma \ref{pureperm}, Dorian Croitoru and Steven Sam for
helpful discussions and reading the draft carefully. I am also very grateful to the anonymous reviewer for very helpful comments and instructions.

\bigskip

\filbreak \noindent Nan Li\\
Department of Mathematics\\
Massachusetts Institute of Technology\\
Cambridge, MA 02139\\
{\tt nan@math.mit.edu}

\end{document}